\documentclass[a4paper,12pt]{article}

\renewenvironment{abstract}
               {\list{}{\rightmargin\leftmargin}%
                \item[\hspace*{1cm}\small\textbf{Abstract ---}]\relax}
               {\endlist}

\usepackage[top=2cm, bottom=2cm, left=2cm, right=2cm]{geometry}
\usepackage[fleqn]{amsmath}
\usepackage{amssymb}
\usepackage{enumerate}
\usepackage{enumitem}
\usepackage{amsthm,bm}
\usepackage{appendix}
\usepackage{subfigure}
\usepackage{sidecap}
\usepackage{float}
\usepackage[pdftex]{graphicx}
\usepackage{mathrsfs}
\newtheorem{Theorem}{Theorem}[section]
\newtheorem{Definition}[Theorem]{Definition}
\newtheorem{Remark}[Theorem]{Remark}
\newtheorem{Proposition}[Theorem]{Proposition}
\newtheorem{Axiom}[Theorem]{Axiom}
\newtheorem{Example}[Theorem]{Example}
\newtheorem{Conjecture}[Theorem]{Conjecture}
\newtheorem{Notation}[Theorem]{Notation}
\newtheorem{Rule}[Theorem]{Inference Rule}
\def\rotatechar#1{\rotatebox[origin=c]{180}{#1}}
\begin{document}

\title{\large\bf A marriage of category theory and set theory:\\ a finitely axiomatized nonstandard first-order theory implying ZF}

\author{
        Marcoen J.T.F. Cabbolet\footnote{e-mail: Marcoen.Cabbolet@vub.ac.be}\\
        \small{\textit{Center for Logic and Philosophy of Science, Vrije Universiteit Brussel}}\\
        }
\date{}

\maketitle

\begin{abstract}\small
It is well known that ZFC, despite its usefulness as a foundational theory for mathematics, has two unwanted features: it cannot be written down explicitly due to its infinitely many axioms, and it has a countable model due to the L\"{o}wenheim-Skolem theorem. This paper presents the axioms one has to accept to get rid of these two features. For that matter, some twenty axioms are formulated in a nonstandard first-order language with countably many constants: to this collection of axioms is associated a universe of discourse consisting of a class of objects, each of which is a set, and a class of arrows, each of which is a function. The axioms of ZF are derived from this finite axiom scheme, and it is shown that it does not have a countable model---if it has a model at all, that is. Furthermore, the axioms of category theory are proven to hold: the present universe may therefore serve as an ontological basis for category theory. However, it has not been investigated whether any of the soundness and completeness properties hold for the present theory: the inevitable conclusion is therefore that only further research can establish whether the present results indeed constitute an advancement in the foundations of mathematics.
\end{abstract}

\section{Introduction}
Zermelo-Fraenkel set theory, which emerged from the axiomatic set theory proposed by Zermelo in \cite{Zermelo} by implementing improvements suggested independently by Skolem and Fraenkel, is without a doubt the most widely accepted foundational theory for mathematics. However, it has two features that we may call `unwanted' or `pathological'. Let us, in accordance with common convention, use ZFC to denote the full theory, ZF for the full theory minus the axiom of choice (AC), and let us use ZF(C) in statements that are to hold for both ZF and ZFC; the first pathological feature is then that the number of axioms of ZF(C) is infinite, as a result of which ZF(C) cannot be written down explicitly. The second is that ZF(C) has a countable model (if it has a model at all), which is a corollary of the downward L\"{o}wenheim-Skolem theorem \cite{Lowenheim,Skolem}: as a result, we have to swallow that there is a model of ZF(C) in which the powerset of the natural numbers is countable. The purpose of this paper is to present the axioms one has to accept such that the axioms of ZF can be derived from the new axioms---we omit a discussion of AC---but such that these two pathological features are \emph{both} removed. The collection of these new axioms will henceforth be referred to with the symbol $\mathfrak{T}$ (a Gothic `T'): it is thus the case that $\mathfrak{T}$ has a finite number of non-logical axioms of finite length. The remainder of this section is intended as a global introduction to $\mathfrak{T}$---the purpose is not, however, to review every idea ever published on this topic.

First of all, at least one set theory that can be finitely axiomatized has already been suggested, namely Von Neumann-G\"{o}del-Bernays set theory (NGB)---NGB is provably a mere conservative extension of ZF, and as such ``it seems to be no stronger than ZF'' \cite{Mendelson}. However, NGB shares the second of the two aforementioned unwanted features with ZF. In that regard, Von Neumann has been quoted stating
\begin{quote}
  ``At present we can do no more than note that we have one more reason here to entertain reservations about set theory and that for the time being no way of rehabilitating this theory is known.'' \cite{VanDalen}
\end{quote}
That said, a finite theory in the same language as ZF (without extra objects) and \emph{as strong as} ZF has already been proved impossible by Montague \cite{Montague}: the present finitely axiomatized theory $\mathfrak{T}$ therefore entails a rather drastic departure from the language and ontology of ZF. Earlier a radical departure has already been suggested by Lawvere, who formulated a theory of the category of sets: here the $\in$-relation has been defined in terms of the primitive notions of category theory, that is, in terms of mappings, domains and codomains \cite{Lawvere}. The present theory $\mathfrak{T}$, however, is more of a ``marriage''--by lack of a better term---between set theory and category theory: the $\in$-relation is maintained as an atomic expression, while the notion of a category is built in as the main structural element. So regarding the philosophical position on the status of category theory, here neither Lawvere's position is taken, that category theory provides the foundation for mathematics \cite{Lawvere2}, nor Mayberry's position that category theory requires set theory as a foundation \cite{Mayberry}, nor Landry's position that category theory provides (all of) the language for mathematics \cite{Landry}: instead, the present position is that category theory is \emph{incorporated} in the foundations.

Category theory being incorporated means that the universe of (mathematical) discourse is not Cantor's paradise of sets, but a category:

\begin{Definition}\label{def:universe}\rm
The \textbf{universe of discourse} is a category $\mathscr{C}$ consisting of
\begin{enumerate}[label=(\roman*)]
\item a proper class of objects, each of which is a \textbf{set};
\item a proper class of arrows, each of which is a \textbf{function}.
\end{enumerate}
$\Box$
\end{Definition}
\noindent As to the meaning of the term ‘universe of discourse’ in Def. \ref{def:universe}, the following quote from a standard textbook is interesting:
\begin{quote}
``Instead of having to say \emph{the entities to which the predicates might in principle apply}, we can make things easier for ourselves by collectively calling these entities the \emph{universe of discourse}.'' \cite{Gamut}
\end{quote}
So, a theory \emph{without} a universe of discourse is nothing but a scheme of meaningless strings of symbols, which are its axioms; an inference rule is then nothing but a rule by which (other) meaningless strings of symbols can be ``deduced'' from those axioms. Such a notion of `theory' might be acceptable from the perspective of a formalist, but here the position is taken that theories c.q. strings of symbols without meaning are \emph{not interesting}. So, to the present theory $\mathfrak{T}$ is associated a Platonic universe of discourse---in casu the category $\mathscr{C}$ of Def. \ref{def:universe}---which we can think of as being made up of the things that satisfy the axioms of $\mathfrak{T}$. But that does not mean that for every thing in the universe of discourse a constant needs to be included in the formal language: the vocabulary contains only countably many constants. But $\mathfrak{T}$ has been formulated with an \emph{intended model} in mind: its universe is then a ``Platonic imitation'' of the universe of discourse. Below we briefly elaborate on the universe of discourse using set-builder notation: strictly speaking this is not a part of the formal language, but given its status as a widely used tool for the description of sets it is suitable for an \emph{informal} introductory exposition that one can hold in the back of one's mind.

For starters, the primitive notion of a set is, of course, that it is an object made up of elements: a thing $\alpha$ being an element of a set $S$ is formalized by an irreducible $\in$-relation, that is, by an atomic expression of the form $\alpha \in S$. The binary predicate `$\in$' is part of the language for $\mathfrak{T}$: there is no need to express it in the language of category theory, since that does not yield a simplification. (Hence the language of the present theory is not reduced to the language of category theory.) In ZF we then have the adage `everything is a set', meaning that if we have $x \in y$, then $x$ is a set too. Here, however, that adage remains valid in this proper class of objects only to the extent that all the \emph{objects} are sets---the adage does not hold for all elements of all sets. That is, we will assume that every object of the category is either the empty set or an object that contains elements, but an element of an object---if any---can either be the empty set, or again an object made up of elements, or a function: a function is then \underline{not} a set. A number of constructive set-theoretical axioms then describe in terms of the $\in$-relation which sets there are \emph{at least} in this proper class of sets; these axioms are very simple theorems of ZF that hardly need any elaboration. 

As to the notion of a function, in the framework of ZF a function is \emph{identified} with its graph. However, as hinted at above, here we reject that set-theoretical reduction. First of all, functions are \emph{objects sui generis}. If we use  simple symbols like $X$ and $Y$ for sets, then a composite symbol $f_X$---to be pronounced ``f-on-X''---can be used for a function on $X$: in the intended model, a function is a thing $f_X$ where $f$ stands for the graph of the function and $X$ for its domain. To give an example, let the numbers 0 and 1 be defined as sets, e.g. by $0 := \emptyset$ and $1 := \{\emptyset\}$, and let, for sets $x$ and $y$, a two-tuple $\langle x,y \rangle$ be defined as a set, e.g. by $\langle x,y \rangle := \{x, \{x, y\}\}$; then the composite symbol $\{\langle 0, 1\rangle,\langle 1,0\rangle\}_{\{0,1\}}$ refers to the function on the set $\{0,1\}$ whose graph is the set $\{\langle 0, 1\rangle,\langle 1,0\rangle\}$. However, we have $f_X \neq Y$ for any function on any domain $X$ and for any set $Y$, so
\begin{equation}
\{\langle 0, 1\rangle,\langle 1,0\rangle\}_{\{0,1\}} \neq \{\langle 0, 1\rangle,\langle 1,0\rangle\}
\end{equation}
That is, the function $\{\langle 0, 1\rangle,\langle 1,0\rangle\}_{\{0,1\}}$ is not identical to its graph $\{\langle 0, 1\rangle,\langle 1,0\rangle\}$---nor, in fact, to any other set $Y$. At this point one might be inclined to think that the whole idea of functions on a set as objects sui generis is superfluous, and should be eliminated in favor of the idea that functions are identified with their graphs (which are sets). That, however, has already been tried in an earlier stage of this investigation: it turned out to lead to unsolvable difficulties with the interpretation of the formalism, cf. \cite{CabboletWithdrawn}. The crux is that the main constructive axiom of the theory---here a \emph{constructive axiom} is an axiom that, when certain things are given (e.g. one or two sets or a set and a predicate), states the existence of a uniquely determined other thing \cite{Bernays}---`produces' things referred to by a symbol $F_X$: one gets into unsolvable difficulties if one tries to interpret these things as sets.

But functions are not only different things than sets. Contrary to a set, a function in addition \emph{has} a domain and a codomain---both are sets---and it also \emph{does something}: namely, it maps every element in its domain to an element in its codomain. This first aspect, that a function `has' a domain and a codomain, can be expressed in the language of category theory: an atomic formula $f_X:Y\twoheadrightarrow Z$ expresses that the function $f_X$ has domain $Y$ and codomain $Z$. In accordance with existing convention, the two-headed arrow `$\twoheadrightarrow$' expresses that $f_X$ is a surjection: the codomain is \emph{always} the set of the images of the elements of the domain under $f_X$. 
Since such an expression $f_X:Y\twoheadrightarrow Z$ is irreducible in the present framework, it requires some function-theoretical axioms to specify when such an atomic formula is true and when not. For example, for the function $\{\langle 0, 1\rangle,\langle 1,0\rangle\}_{\{0,1\}}$ discussed above we have
\begin{equation}
\{\langle 0, 1\rangle,\langle 1,0\rangle\}_{\{0,1\}} : {\{0,1\}} \twoheadrightarrow {\{0,1\}}
\end{equation}
while other expressions $\{\langle 0, 1\rangle,\langle 1,0\rangle\}_{\{0,1\}} : Y \twoheadrightarrow Z$ are false. Again, at this point one might be inclined to think that these expression $f_X:Y\twoheadrightarrow Z$ are superfluous, because the notation $f_X$ already indicates that $X$ is the domain of $f_X$. The crux, however, is that these expressions are \emph{essential} to get to a finite axiomatization: we may, then, have the opinion that it is obvious from the notation $f_X$ that $X$ is the domain of $f_X$, but only an expression $f_X:X\twoheadrightarrow Z$ expresses this fact---it is, thus, an axiom of the theory that $f_X:Y\twoheadrightarrow Z$ is only true if $Y = X$. In particular, this expression is not true if $Y \varsubsetneq X$.

The second aspect, that for any set $X$ any function $f_X$ `maps' an element $y$ of its domain to an element $z$ of its codomain is expressed by another atomic formula $f_X:y\mapsto z$. In the framework of ZF this is just a notation for $\langle y,z \rangle\in f_X$, but in the present framework this is also an irreducible expression: therefore, it requires some more function-theoretical axioms to specify when such an atomic formula is true and when not. The idea, however, is this: \textbf{given} a set $X$ and a function $f_X$, precisely one expression $f_X:y\mapsto z$ is true for each element $y$ in the domain of $f_X$. E.g. we have
\begin{gather}
\{\langle 0, 1\rangle,\langle 1,0\rangle\}_{\{0,1\}} : 0 \mapsto 1\\
\{\langle 0, 1\rangle,\langle 1,0\rangle\}_{\{0,1\}} : 1 \mapsto 0
\end{gather}
for the above function $\{\langle 0, 1\rangle,\langle 1,0\rangle\}_{\{0,1\}}$.

All the above can be expressed with a dozen and a half very simple axioms---these can, in fact, all be reformulated in the framework of ZF. The present axiomatic scheme does, however, contain `new mathematics' in the form of the second axiom of a pair of constructive, function-theoretical axioms. The first one is again very simple, and merely states that given any two singletons $X=\{x\}$ and $Y=\{y\}$, there exists an `ur-function' that maps the one element $x$ of $X$ to the one element $y$ of $Y$. An ur-function is thus a function with a singleton domain and a singleton codomain. The above function $\{\langle 0, 1\rangle,\langle 1,0\rangle\}_{\{0,1\}}$ is thus not an ur-function, but the function referred to by the symbol  $\{\langle 0, 1\rangle\}_{\{0\}}$ is: we have
\begin{gather}
\{\langle 0, 1\rangle\}_{\{0\}}: \{0\} \twoheadrightarrow \{1\} \\
\{\langle 0, 1\rangle\}_{\{0\}}: 0 \mapsto 1
\end{gather}
That said, the second of the pair of constructive, function-theoretical axioms is a new mathematical principle: it states that given any family of ur-functions $f_{\{j\}}$ indexed in a set $Z$, there exists a sum function $F_Z$ such that the sum function maps an element $j$ of $Z$ to the same image as the corresponding ur-function $f_{\{j\}}$. The formulation of this principle requires, however, a new nonstandard concept that can be called a `multiple quantifier'. This new concept can be explained as follows. Suppose we have defined the natural numbers $0, 1, 2, \ldots$ as sets, suppose the singletons $\{0\}, \{1\}, \{2\}, \ldots$ exist as well, and suppose we also have the set of natural numbers $\omega = \{0,1,2,\ldots\}$. We can, then, consider variables $f_{\{0\}}, f_{\{1\}}, f_{\{2\}}, \ldots$ ranging over all ur-functions on a singleton of a natural number as indicated by the subscript---so, $f_{\{0\}}$ is a variable that ranges over all ur-functions on the singleton of $0$. In the standard first-order language of ZF we have the possibility to quantify over all ur-functions on a singleton $\{a\}$ by using a quantifier $\forall f_{\{a\}}$, and we have the possibility to use any finite number of such quantifiers in a sentence. E.g. we can have a formula
\begin{equation}\label{eq:1a}
\forall f_{\{0\}}\forall f_{\{1\}} \forall f_{\{2\}}\Psi
\end{equation}
meaning: for all ur-functions on ${\{0\}}$ and for all ur-functions on ${\{1\}}$ and for all ur-functions on ${\{2\}}$, $\Psi$. We can then introduce a new notation by the following postulate of meaning:
\begin{equation}\label{eq:1b}
{(\forall f_{\{j\}})}_{j \in \{0, 1, 2\} } \Psi \Leftrightarrow \forall f_{\{0\}}\forall f_{\{1\}} \forall f_{\{2\}}\Psi
\end{equation}
This new formula can be read as: for any family of ur-functions indexed in $\{0, 1, 2\}$, $\Psi$. So $(\forall f_{\{i\}})_{i \in \{0, 1, 2\} }$ is then a multiple quantifier that, in this case, is equivalent to three quantifiers in standard first-order language. The next step is now that we lift the restriction that a multiple quantifier has to be equivalent to a finite number of standard quantifiers: with that step we enter into nonstandard territory. In the nonstandard language of $\mathfrak{T}$ we can consider a formula like
\begin{equation}\label{eq:1c}
{(\forall f_{\{i\}})}_{i \in \omega } \Psi
\end{equation}
The multiple quantifier is equivalent to a sequence $\forall f_{\{0\}} \forall f_{\{1\}} \forall f_{\{2\}} \cdots$ of infinitely many standard quantifiers in this case. But in the present framework the constant $\omega$ in formula (\ref{eq:1c}) can be replaced by \emph{any} constant, yielding nonstandard multiple quantifiers equivalent to an uncountably infinite number of standard quantifiers.

To yield meaningful theorems, the subformula $\Psi$ of formula (\ref{eq:1c}) has to be open in infinitely many variables $f_{\{0\}}, f_{\{1\}}, f_{\{2\}}, \ldots$, each of which ranges over all ur-functions on a singleton of a natural number. This is achieved by placing a \emph{conjunctive operator} ${\bigwedge}_{i \in \omega}$ in front of a standard first-order formula $\Psi(f_{\{i\}})$ that is open in a composite variable $f_{\{i\}}$, yielding an expression
\begin{equation}\label{eq:2a}
{\bigwedge}_{i \in \omega} \Psi(f_{\{i\}})
\end{equation}
\emph{Syntactically} this is a formula of finite length, but \emph{semantically} it is the conjunction of a countably infinite family of formulas $\Psi(f_{\{0\}}), \Psi(f_{\{1\}}), \Psi(f_{\{2\}}), \ldots$. Together with the multiple quantifier from formula (\ref{eq:1c}) it yields a \emph{sentence} ${(\forall f_{\{i\}})}_{i \in \omega } {\bigwedge}_{i \in \omega} \Psi(f_{\{i\}})$: semantically it contains a bound occurrence of the variables $f_{\{0\}}, f_{\{1\}}, f_{\{2\}}, \ldots$. The nonstandard `sum function axiom' constructed with these formal language elements is so powerful that it allows to derive the infinite schemes SEP and SUB of ZF from just a finite number of axioms.

That said, the literature already contains a plethora of so-called \emph{infinitary logics}; see \cite{Bell} for a short, general overview. In the framework of such an infinitary logic, one typically can form infinitary conjunctions of a collection $\Sigma$ of standard first-order formulas indexed by some set $S$. So, if $\Sigma = \{\phi_j \ | \ j \in S\}$ is such a collection of formulas, then a formula $\bigwedge \Sigma$ stands for $\wedge_{j\in S} \phi_j$, which resembles formula (\ref{eq:2a}). That way, one can, for example, form the conjunction of all the formulas of the SEP scheme of ZF: that yields a finite axiomatization of set theory. However, such an infinitary conjunction cannot be written down explicitly: the string of symbols `$\wedge_{j\in S} \phi_j$' is an informal abbreviation that is not part of the formal language---that is, it is not a well-formed formula. In the present case, however, we are only interested in well-formed formulas of finite length, which can be written down explicitly. This is achieved by building restrictions in the definition of the syntax, so that  a conjunctive operator like ${\bigwedge}_{x \in \omega}$ only forms a well-formed formula if it is put in front of an atomic expression of the type $f_{\{x\}}:x\mapsto t(x)$ where both $f_{\{x\}}$ and $t(x)$ are terms with an occurrence of the same variable $x$ that occurs in the conjunctive operator: that way typographically finite expressions obtain that semantically are infinitary conjunctions. So, the language for our theory $\mathfrak{T}$ is much more narrowly defined than the language of (overly?) general infinitary logics: as it turns out, that suffices for our present aims.

That brings us to the next point, which is to discuss the (possible) practical use of the new finite theory $\mathfrak{T}$  as a foundational theory for mathematics. The practical usefulness of the scheme lies therein that it (i) provides an easy way to construct sets, and (ii) that categories like \textbf{Top}, \textbf{Mon}, \textbf{Grp}, etc, which are subjects of study in category theory, can be viewed as subcategories of the category of sets and functions of Def. \ref{def:universe}, thus providing a new approach to the foundational problem identified in \cite{Cabbolet}. While that latter point (ii) hardly needs elaboration, the former (i) does. The crux is that one doesn't need to apply the nonstandard sum function axiom directly: what one uses is the theorem---or rather: the theorem scheme---that given any set $X$, we can construct a new function on $X$ by giving a function prescription. So, this is a philosophical nuance: in ZF one constructs a new object with the $\in$-relation, but in the present framework one can construct a new function $f_X$ not with the $\in$-relation, but by simply defining which expressions $f_X: y \mapsto z$ are true for the elements $y$ in the domain. On account of the sum function axiom, it is then a guarantee that the function $f_X$ exists. So, given any set $X$ one can simply give a defining function prescription
\begin{equation}
f_X: y \overset{\rm def}{\longmapsto} \imath z \Phi (y, z)
\end{equation}
(where the iota-term $\imath z\Phi (y, z)$ denotes the unique thing $z$ for which the functional relation $\Phi (y, z)$ holds) and it is then a guarantee that $f_X$ exists. And having constructed the function, we have implicitly constructed its graph and its image set. Ergo, giving a function prescription \emph{is} constructing a set. The nonstandard axiom, which may be cumbersome to use directly, stays thus in the background: one uses the main theorem of $\mathfrak{T}$.

What remains to be discussed is that $\mathfrak{T}$, as mentioned in the first paragraph, entails the axioms of ZF and does not have a countable model. As to the first-mentioned property, it will be proven that the infinite axiom schemes SEP and SUP of ZF, translated in the formal language of $\mathfrak{T}$, can be derived from the finitely axiomatized theory $\mathfrak{T}$: that provides an argument for considering $\mathfrak{T}$ to be \emph{not weaker} than ZF. Secondly, even though the language of $\mathfrak{T}$ has only countably many constants, it will be shown that the validity of the nonstandard sum function axiom in a model $\mathcal{M}$ of $\mathfrak{T}$ has the consequence that the downward L\"{o}wenheim-Skolem theorem does not hold: if $\mathfrak{T}$ has a model $\mathcal{M}$, then $\mathcal{M}$ is uncountable. This is a significant result that does not hold in the framework of ZF: it provides, therefore, an argument for considering $\mathfrak{T}$ to be \emph{stronger} than ZF.\\
\ \\
That concludes the introductory discussion. The remainder of this paper is organized as follows. The next section axiomatically introduces this finitely axiomatized nonstandard theory $\mathfrak{T}$. The section thereafter discusses the theory $\mathfrak{T}$ (i) by deriving its main theorem, (ii) by deriving the axiom schemes SEP and SUB of ZF, (iii) by showing that the downward L\"{o}wenheim-Skolem theorem does not hold for the nonstandard theory $\mathfrak{T}$, (iv) by showing the axioms of category theory hold for the class of arrows of the universe of discourse meant in Def. \ref{def:universe}, and (v) by addressing the main concerns for inconsistency. The final section states the conclusions.

\section{Axiomatic introduction}\label{sect:2}

\subsection{Formal language}

First of all a remark. In standard first-order logic, the term `quantifier' refers \emph{both} to the logical symbols `$\forall$' and `$\exists$' \emph{and} to combinations like $\forall x$, $\exists y$ consisting of such a symbol and a variable. While that may be unproblematic in standard logic, here the logical symbols `$\forall$' and `$\exists$' are applied in \emph{different} kinds of quantifiers. Therefore, to avoid confusion we will refer to the symbol `$\forall$' with the term `universal quantification symbol', and to the symbol `$\exists$' with the term `existential quantification symbol'. 

\begin{Definition}\label{def:vocabulary}\rm
The \textbf{vocabulary} of the language $L_\mathfrak{T}$ of $\mathfrak{T}$ consists of the following:
\begin{enumerate}[label=(\roman*)]
\item the constants $\emptyset$ and $\bm{\omega}$, to be interpreted as the empty set and the first infinite ordinal;
\item the constants $1_\emptyset$, to be interpreted as the inactive function;
\item simple variables $x, X, y, Y, \ldots$ ranging over sets;
\item for any constant ${\rm\hat{\mathbf{X}}}$ referring to an individual set, composite symbols $f_{{\rm\hat{\mathbf{X}}}}, g_{{\rm\hat{\mathbf{X}}}}, \ldots$ with an occurrence of the constant ${\rm\hat{\mathbf{X}}}$ as the subscript are simple variables ranging over functions on that set ${\rm\hat{\mathbf{X}}}$;
\item for any simple variable $X$ ranging over sets, composite symbols $f_X, g_X, \dots$ with an occurrence of the variable $X$ as the subscript are composite variables ranging over functions on a set $X$;
\item simple variables $\alpha, \beta, \ldots$ ranging over all things (sets and functions);
\item the binary predicates `$\in$' and `=', and the ternary predicates `$(.):(.)\twoheadrightarrow (.)$' and `$(.):(.)\mapsto (.)$';
\item the logical connectives of first-order logic $\neg, \wedge, \vee, \Rightarrow, \Leftrightarrow$;
\item the universal and existential quantification symbols $\forall, \exists$;
\item the brackets `(' and `)'.
\end{enumerate}
$\Box$
\end{Definition}


\begin{Remark}\rm To distinguish between the theory $\mathfrak{T}$ and its intended model, boldface symbols with a hat like $\mathrm{\hat{\mathbf{X}}}$, $\hat{\bm{\alpha}}$, etc. will be used to denote constants in the vocabulary of the theory $\mathfrak{T}$, while underlined boldface symbols like $\mathrm{\underline{\mathbf{X}}}$, $\underline{\bm{\alpha}}$, etc., will be used to denote individuals in the intended model of $\mathfrak{T}$. \hfill$\Box$
\end{Remark}

\begin{Definition}\label{def:syntax}\rm
The \textbf{syntax} of the language $L_\mathfrak{T}$ is defined by the following clauses:
\begin{enumerate}[label=(\roman*)]
\item if $t$ is a constant, or a simple or a composite variable, then $t$ is a term;
\item if $t_1$ and $t_2$ are terms, then $t_1 = t_2$ and $t_1 \in t_2$ are atomic formulas;
\item if $t_1$, $t_2$, and $t_3$ are terms, then $t_1:t_2\twoheadrightarrow t_3$ and $t_1:t_2\mapsto t_3$ are atomic formulas;
\item if $\Phi$ and $\Psi$ are formulas, then $\neg\Phi$, $(\Phi \wedge \Psi)$, $(\Phi \vee \Psi)$, $(\Phi \Rightarrow \Psi)$, $(\Phi \Leftrightarrow \Psi)$ are formulas;
\item if $\Psi$ is a formula and $t$ a simple variable ranging over sets, over all things, or over functions on a constant set, then $\forall t \Psi$ and $\exists t \Psi$ are formulas;
\item if $X$ and $f_{{\rm\hat{\mathbf{X}}}}$ are simple variables ranging respectively over sets and over functions on the set ${\rm\hat{\mathbf{X}}}$, $f_X$ a composite variable with an occurrence of $X$ as the subscript, and $\Psi$ a formula with an occurrence of a quantifier $\forall f_{{\rm\hat{\mathbf{X}}}}$ or $\exists f_{{\rm\hat{\mathbf{X}}}}$ but with no occurrence of $X$, then $\forall X [X \backslash {\rm\hat{\mathbf{X}}}] \Psi$ and $\exists X [X \backslash {\rm\hat{\mathbf{X}}}] \Psi$ are formulas.
\end{enumerate}
$\Box$
\end{Definition}

\begin{Remark}\rm Regarding clause (vi) of Def. \ref{def:syntax}, $[u \backslash t]\Psi$ is the formula obtained from $\Psi$ by replacing $t$ everywhere by $u$. This definition will be used throughout this paper. Note that if $\Psi$ is a formula with an occurrence of the simple variable $f_{{\rm\hat{\mathbf{X}}}}$, then $[X\backslash {{\rm\hat{\mathbf{X}}}}]\Psi$ is a formula with an occurrence of the composite variable $f_X$. \hfill$\Box$
\end{Remark}
\begin{Definition}\label{def:syntax-elements}\rm
The language $L_\mathfrak{T}$ contains the following special language elements:
\begin{enumerate}[label=(\roman*)]
\item if $t$ is a simple variable ranging over sets, over all things, or over functions on a constant set, then $\forall t$ and $\exists t$ are \textbf{quantifiers with a simple variable};
\item if $f_X$ is a composite variable, then $\forall f_X$ and $\exists f_X$ are a \textbf{quantifiers with a composite variable}.
\end{enumerate}
A sequence like $\forall X\forall f_X$ can be called a \textbf{double quantifier}.\hfill $\Box$
\end{Definition}

\noindent The scope of a quantifier is defined as usual: note that a quantifier with a composite variable can only occur in the scope of a quantifier with a simple variable. A free occurrence and a bounded occurrence of a simple variable is also defined as usual; these notions can be simply defined for formulas with a composite variable.

\begin{Definition}\label{def:double-quantifier}\rm
Let $f_X$ be a composite variable with an occurrence of the simple variable $X$; then
\begin{enumerate}[label=(\roman*)]
\item an occurrence of $f_X$ in a formula $\Psi$ is \textbf{free} if that occurrence is neither in the scope of a quantifier with the composite variable $f_X$ nor in the scope of a quantifier with the simple variable $X$;
\item an occurrence of $f_X$ in a formula $\Psi$ is \textbf{bounded} if that occurrence is in the scope of a quantifier with the composite variable $f_X$.
\end{enumerate}
A \textbf{sentence} is a formula with no free variables---simple or composite. A formula is \textbf{open} in a variable if there is a free occurrence of that variable. A formula that is open in a composite variable $f_X$ is also open in the simple variable $X$. \hfill$\Box$
\end{Definition}

\begin{Definition}\label{def:classical-semantics}\rm
The \textbf{semantics} of any sentence without a quantifier with a composite variable is as usual. Furthermore,
\begin{enumerate}[label=(\roman*)]
\item a sentence $\forall X \Psi$ with an occurrence of a quantifier $\forall f_X$ or $\exists f_X$ with the composite variable $f_X$ is valid in a model $\mathcal{M}$ if and only if for every assignment $g$ that assigns an individual set $g(X) = {\rm\underline{\mathbf{X}}}$ in $\mathcal{M}$ as a value to the variable $X$, the sentence $[{\rm\underline{\mathbf{X}}}\backslash X]\Psi$ is valid in $\mathcal{M}$;
\item the sentence $[{\rm\underline{\mathbf{X}}}\backslash X]\Psi$ obtained in clause (i) is a sentence without a quantifier with a composite variable, hence with usual semantics.
\end{enumerate}
The semantics of a sentence $\exists X \Psi$ with an occurrence of a quantifier $\forall f_X$ or $\exists f_X$ is left as an exercise.\hfill $\Box$
\end{Definition}

\noindent After the introduction of the `standard' first-order axioms of $\mathfrak{T}$, the language $L_\mathfrak{T}$ will be extended to enable the formulation of the desired nonstandard formulas.

\subsection{Set-theoretical axioms}
Below the set-theoretical axioms are listed; these are all standard first-order formulas. Due to the simplicity of the axioms, comments are kept to a bare minimum.

\begin{Axiom}\label{ax:EXT}\rm Extensionality Axiom for Sets (EXT): two sets $X$ and $Y$ are identical if they have the same things (sets and functions) as elements.
\begin{equation}\label{eq:}
\forall X \forall Y (X = Y \Leftrightarrow \forall \alpha (\alpha \in X \Leftrightarrow \alpha \in Y))
\end{equation}
$\Box$
\end{Axiom}

\begin{Axiom}\label{ax:FnotS}\rm For any set $X$, any function $f_X$ is not identical to any set $Y$:
\begin{equation}
\forall X \forall f_X \forall Y (f_X \neq Y)
\end{equation}
$\Box$
\end{Axiom}

\begin{Axiom}\label{ax:SnoDom}\rm A set $X$ has no domain or codomain, nor does it map any thing to an image:
\begin{equation}
\forall X \forall\alpha\forall\beta(X:\alpha\not\twoheadrightarrow\beta \wedge X:\alpha\not\mapsto\beta)
\end{equation}
$\Box$
\end{Axiom}

\noindent These latter two axioms establish that sets are different from functions on a set, and do not have the properties of functions on a set.

\begin{Axiom}\label{ax:EMPTY}\rm Empty Set Axiom (EMPTY): there exists a set $X$, designated by the constant $\emptyset$, that has no elements.
\begin{equation}\label{eq:}
\exists X(X = \emptyset \wedge \forall\alpha(\alpha \not \in X))
\end{equation}
$\Box$
\end{Axiom}

\begin{Axiom}\label{ax:PAIR}\rm Axiom of Pairing (PAIR): for every thing $\alpha$ and every thing $\beta$ there exists a set $X$ that has precisely the things $\alpha$ and $\beta$ as its elements.
\begin{equation}\label{eq:}
\forall \alpha \forall \beta \exists X \forall \gamma (\gamma\in X \Leftrightarrow \gamma = \alpha \vee \gamma = \beta)
\end{equation}
$\Box$
\end{Axiom}

\begin{Remark}\rm Using set-builder notation, the empty set can be interpreted as the individual $\{\}$ in the intended model. Furthermore, given individual things $\underline{\bm{\alpha}}$ and $\underline{\bm{\beta}}$ in the universe of the intended model, the pair set of $\underline{\bm{\alpha}}$ and $\underline{\bm{\beta}}$ can then be identified with the individual $\{\underline{\bm{\alpha}}, \underline{\bm{\beta}}\}$. Note that this is a singleton if $\underline{\bm{\alpha}}=\underline{\bm{\beta}}$. \hfill$\Box$
\end{Remark}

\begin{Definition}\label{def:singleton}\rm (Extension of vocabulary of $L_\mathfrak{T}$.)\\ If $t$ is a term, then $(t)^+$ is a term, to be called ``the singleton of $t$''---which may be written as $t^+$ if no confusion arises. In particular, if $\alpha$ is a variable ranging over all things and $x$ a variable ranging over all sets, then $\alpha^+$ is a variable ranging over all singletons and $x^+$ a variable ranging over all singletons of sets.
We thus have
\begin{equation}
  \forall\alpha\forall\beta(\beta = \alpha^+ \Leftrightarrow \exists X(\beta = X \wedge \forall\gamma(\gamma\in X \Leftrightarrow\gamma = \alpha)))
\end{equation}
Likewise for $x^+$. \hfill$\Box$
\end{Definition}

\begin{Notation}\label{not:ZermeloOrdinals}\rm
On account of Def. \ref{def:singleton}, the constants $\emptyset^{+}, \emptyset^{++}, \emptyset^{+++}, \ldots$ are contained in the language $L_{\mathfrak{T}}$ . Therefore, we can introduce the (finite) \textbf{Zermelo ordinals} at this point as a notation for these singletons:
\begin{equation}
\left\{  \begin{array}{l}
           0:= \emptyset  \\
           1:= \emptyset^{+} \\
           2:= \emptyset^{++}\\
           \vdots
         \end{array}
\right.
\end{equation}
According to the literature, the idea stems from unpublished work by Zermelo in 1916 \cite{Levy}.\hfill$\Box$
\end{Notation}
\begin{Axiom}\label{ax:SUM}\rm Sum Set Axiom (SUM): for every set $X$ there exists a set $Y$ made up of the elements of the elements of $X$.
\begin{equation}\label{eq:}
\forall X \exists Y \forall \alpha(\alpha \in Y \Leftrightarrow \exists Z(Z \in X \wedge \alpha \in Z))
\end{equation}
$\Box$
\end{Axiom}

\begin{Remark}\rm Given an individual set ${\rm\underline{\mathbf{X}}}$ in the universe of the intended model, the sum set of ${\rm\underline{\mathbf{X}}}$ can be denoted by the symbol $\bigcup {\rm\underline{\mathbf{X}}}$ and, using set-builder notation, be identified with the individual $\{\alpha \ | \ \exists Z \in {\rm\underline{\mathbf{X}}}(\alpha\in Z)\}$. \hfill$\Box$
\end{Remark}

\begin{Axiom}\label{ax:POW}\rm Powerset Axiom (POW): for every set $X$ there is a set $Y$ made up of the subsets of $X$.
\begin{equation}\label{eq:}
\forall X \exists Y \forall \alpha (\alpha \in Y \Leftrightarrow \exists Z(Z \subset X \wedge \alpha = Z))
\end{equation}
$\Box$
\end{Axiom}

\begin{Remark}\rm Given an individual set ${\rm\underline{\mathbf{X}}}$ in the universe of the intended model, the powerset of ${\rm\underline{\mathbf{X}}}$ can be denoted by the symbol $\mathcal{P}({\rm\underline{\mathbf{X}}})$ and, using set-builder notation, be identified with the individual $\{x \ | \ x \subset{\rm\underline{\mathbf{X}}}\}$. \hfill$\Box$
\end{Remark}

\begin{Axiom}\label{ax:INF}\rm  Infinite Ordinal Axiom (INF): the infinite ordinal $\bm{\omega}$ is the set of all finite Zermelo ordinals.
\begin{equation}\label{eq:}
0 \in \bm{\omega} \wedge \forall \alpha (\alpha\in \bm{\omega} \Rightarrow \alpha^+ \in \bm{\omega}) \wedge \forall \beta \in \bm{\omega} (\not\exists\gamma \in \bm{\omega}(\beta= \gamma^+)\Leftrightarrow\beta=\emptyset))
\end{equation}
$\Box$
\end{Axiom}

\begin{Remark}\label{rem:omega}\rm The set $\bm{\omega}$ in INF is uniquely determined. In the intended model, the set $\bm{\omega}$ can be denoted by the symbol $\mathbb{N}$ and, using set-builder notation, be identified with the individual $\mathbb{N} := \{\ \{\}, \{\{\}\}, \{\{\{\}\}\}, \ldots \}$. \hfill$\Box$
\end{Remark}

\begin{Axiom}\label{ax:REG}\rm Axiom of Regularity (REG): every nonempty set $X$ contains an element $\alpha$ that has no elements in common with $X$.
\begin{equation}\label{eq:}
\forall X \neq \emptyset \exists \alpha (\alpha \in X \wedge \forall \beta (\beta \in \alpha \Rightarrow \beta \not \in X))
\end{equation}
$\Box$
\end{Axiom}

\begin{Definition}\label{def:TwoTuple}\rm For any things $\alpha$ and $\beta$, the \textbf{two-tuple} $\langle\alpha, \beta\rangle$ is the pair set of $\alpha$ and the pair set of $\alpha$ and $\beta$; using the iota-operator we get
\begin{equation}
\langle\alpha, \beta\rangle:= \imath x(\forall\gamma(\gamma \in x \Leftrightarrow \gamma = \alpha \vee \exists Z(\gamma = Z \wedge \forall \eta(\eta \in Z \Leftrightarrow \eta = \alpha \vee \eta = \beta)))
\end{equation}
$\Box$
\end{Definition}

\noindent A simple corollary of Def. \ref{def:TwoTuple} is that for any things $\alpha$ and $\beta$, the two-tuple $\langle\alpha, \beta\rangle$ always exists. There is, thus, no danger of nonsensical terms involved in the use of the iota-operator in Def. \ref{def:TwoTuple}.

\begin{Remark}\rm Given individual things $\underline{\bm{\alpha}}$ and $\underline{\bm{\beta}}$ in the universe of the intended model, the two-tuple $\langle \underline{\bm{\alpha}}, \underline{\bm{\beta}}\rangle$ can, using set-builder notation, be identified with the individual $\{\underline{\bm{\alpha}}, \{\underline{\bm{\alpha}},\underline{\bm{\beta}}\}\}$.\hfill$\Box$
\end{Remark}

\noindent In principle, these set-theoretical axioms suffice: the function-theoretical axioms in the next section provide other means for the construction of sets.

\subsection{Standard function-theoretical axioms}\label{sect:functions}

\begin{Axiom}\label{ax:EPS}\rm A function $f_X$ on a set $X$ has no elements:
\begin{equation}
\forall X \forall f_X \forall \alpha (\alpha \not \in f_X)
\end{equation}
$\Box$
\end{Axiom}

\begin{Remark}\rm One might think that Ax. \ref{ax:EPS} destroys the uniqueness of the empty set. But that is not true. It is true that a function on a set $X$ and the empty set share the property that they have no elements, but the empty set is the only \textbf{set} that has this property: Ax. \ref{ax:FnotS} guarantees, namely, that a function on a set $X$ is not a set!\hfill$\Box$
\end{Remark}

\begin{Axiom}\label{ax:GEN-F}\rm General Function-Theoretical Axiom (GEN-F): for any nonempty set $X$, any function $f_X$ has a set $Y$ as domain and a set $Z$ as codomain, and maps every element $\alpha$ in $Y$ to a unique image $\beta$:
\begin{equation}
\forall X \forall f_X (X\neq \emptyset \Rightarrow \exists Y\exists Z(f_X: Y \twoheadrightarrow Z \wedge \forall \alpha\in Y \exists! \beta (f_X: \alpha \mapsto \beta)))
\end{equation}
$\Box$
\end{Axiom}

\begin{Axiom}\label{ax:NEGD}\rm For any set $X$, any function $f_X$ has no other domain than $X$:\\
\begin{equation}
\forall X \forall f_X \forall \alpha (\alpha\neq X \Rightarrow \forall \xi (f_X: \alpha \not \twoheadrightarrow \xi))
\end{equation}
$\Box$
\end{Axiom}

\begin{Axiom}\label{ax:NONA}\rm For any set $X$, any function $f_X$ does not take a thing outside $X$ as argument:
\begin{equation}
\forall X \forall f_X\forall \alpha \not\in X \forall \beta(f_X: \alpha \not\mapsto \beta)
\end{equation}
$\Box$
\end{Axiom}

\begin{Remark}\rm Ax. \ref{ax:GEN-F} dictates that precisely one expression $f_X: \alpha \mapsto \beta$ is true for each $\alpha \in X$. This doesn't a priori exclude that such an expression can also be true for another thing $\alpha$ not in $X$. But by Ax. \ref{ax:NONA} this is excluded.\hfill$\Box$
\end{Remark}

\begin{Axiom}\label{ax:NEGC}\rm For any nonempty set $X$ and any function $f_X$, the image set is the only codomain:
\begin{equation}
\forall X \forall f_X (X\neq \emptyset \Rightarrow \forall \beta(f_X:X \twoheadrightarrow \beta \Rightarrow \exists Z(\beta = Z \wedge \forall \gamma(\gamma \in Z \Leftrightarrow \exists \eta\in X(f_X: \eta \mapsto \gamma)))))
\end{equation}
(This is the justification for the use of the two-headed arrow `$\twoheadrightarrow$', commonly used for surjections.) \hfill$\Box$
\end{Axiom}

\begin{Remark}\rm
Note that these first function-theoretical axioms already provide a tool to construct a set: if we can construct a new function $f_X$ on a set $X$ from existing functions (an axiom will be given further below), then these axioms guarantee the existence of a unique codomain made up of all the images of the elements of $X$ under $f_X$. \hfill$\Box$
\end{Remark}

\begin{Remark}\rm
Given an individual set ${\rm\underline{\mathbf{X}}}$ and an individual function ${\rm\underline{\mathbf{f}}}_{\rm\underline{\mathbf{X}}}$ in the universe of the intended model, this unique codomain can be denoted by a symbol ${\rm\underline{\mathbf{f}}}_{\rm\underline{\mathbf{X}}}[{\rm\underline{\mathbf{X}}}]$ or ${\rm cod}({\rm\underline{\mathbf{f}}}_{\rm\underline{\mathbf{X}}})$, and, using set-builder notation, be identified with the individual $\{\beta \ | \ \exists\alpha\in{\rm\underline{\mathbf{X}}}({\rm\underline{\mathbf{f}}}_{\rm\underline{\mathbf{X}}}:\alpha\mapsto\beta)\}$ in the universe of the intended model. Furthermore, given a thing $\underline{\bm{\alpha}}$ in ${\rm\underline{\mathbf{X}}}$, its unique image under ${\rm\underline{\mathbf{f}}}_{\rm\underline{\mathbf{X}}}$ can be denoted by the symbol ${\rm\underline{\mathbf{f}}}_{\rm\underline{\mathbf{X}}}(\underline{\bm{\alpha}})$.\hfill$\Box$
\end{Remark}

\begin{Notation}\rm At this point we can introduce expressions $f_X:X\rightarrow Y$, to be read as ``the function f-on-X is a function from the set $X$ to the set $Y$'', by the postulate of meaning
\begin{equation}
f_X:X\rightarrow Y \Leftrightarrow \exists Z(f_X :X \twoheadrightarrow Z \wedge Z\subset Y)
\end{equation}
This provides a connection to existing mathematical practices. \hfill$\Box$
\end{Notation}

\begin{Axiom}\label{ax:INV}\rm Inverse Image Set Axiom (INV): for any nonempty set $X$ and any function $f_X$ with domain $X$ and any co-domain $Y$, there is for any thing $\alpha$ a set $Z\subset X$ that contains precisely the elements of $X$ that are mapped to $\alpha$ by $f_X$:
\begin{equation}
\forall X \neq \emptyset \forall f_X\forall Y(f_X: X\twoheadrightarrow Y\Rightarrow  \forall \beta\exists Z\forall \alpha(\alpha \in Z \Leftrightarrow \alpha \in X \wedge f_X: \alpha \mapsto \beta))
\end{equation}
$\Box$
\end{Axiom}

\begin{Remark}\rm
Note that INV, in addition to GEN-F, also provides a tool to construct a set: if we have constructed a new function $f_X$ with domain $X$ from existing functions, then with this axiom guarantees that the inverse image set exists of any thing $\beta$ in the codomain of $f_X$. \hfill$\Box$
\end{Remark}

\begin{Remark}\rm
Given an individual set ${\rm\underline{\mathbf{X}}}$, an individual function ${\rm\underline{\mathbf{f}}}_{\rm\underline{\mathbf{X}}}$, and an individual thing $\underline{\bm{\beta}}$ in the universe of the intended model,
the unique inverse image set can be denoted by the symbol ${\rm\underline{\mathbf{f}}}_{\rm\underline{\mathbf{X}}}^{-1}(\underline{\bm{\beta}})$ and can, using set-builder notation, be identified with the individual $\{\alpha \ | \ \alpha\in{\rm\underline{\mathbf{X}}}\wedge{\rm\underline{\mathbf{f}}}_{\rm\underline{\mathbf{X}}}:\alpha\mapsto\underline{\bm{\beta}} \}$ in the universe of the intended model.\hfill$\Box$
\end{Remark}

\begin{Axiom}\label{ax:EXTF}\rm Extensionality Axiom for Functions (EXT-F): for any set $X$ and any function $f_X$, and for any set $Y$ and any function $g_Y$, the function $f_X$ and the function $g_Y$ are identical if and only if their domains are identical and their images are identical for every argument:
\begin{equation}
\forall X \forall f_X \forall Y \forall g_Y(f_X = g_Y \Leftrightarrow X = Y \wedge \forall \alpha\forall\beta (f_X:\alpha\mapsto\beta \Leftrightarrow g_Y:\alpha\mapsto\beta))
\end{equation}
$\Box$
\end{Axiom}

\begin{Axiom}\label{ax:emptyfunction}\rm Inactive Function Axiom (IN-F): there exists a function $f_\emptyset$, denoted by the constant $1_\emptyset$, which has the empty set as domain and codomain, and which doesn't map any argument to any image:
\begin{equation}
\exists f_\emptyset (f_\emptyset = 1_\emptyset \wedge f_\emptyset: \emptyset \twoheadrightarrow \emptyset \wedge \forall \alpha \forall \beta (f_\emptyset: \alpha \not \mapsto \beta))
\end{equation}
$\Box$
\end{Axiom}

\noindent Note that there can be no other functions on the empty set than the inactive function $1_\emptyset$, since the image set is always empty: the atomic expression $f_\emptyset:\emptyset \twoheadrightarrow A$ cannot be true for any nonempty set $A$.

\begin{Axiom}\label{ax:UFA}\rm Ur-Function Axiom (UFA): for any things $\alpha$ and $\beta$ there exists an ur-function $f_{\alpha^+}$ with domain $\alpha^+$ and codomain $\beta^+$ that maps $\alpha$ to $\beta$:\\
\begin{equation}
\forall \alpha \forall \beta \exists f_{\alpha^+} (f_{\alpha^+}: \alpha^+ \twoheadrightarrow \beta^+ \wedge f_{\alpha^+}: \alpha \mapsto \beta)
\end{equation}
$\Box$
\end{Axiom}

\begin{Remark}\rm
Given individual things $\underline{\bm{\alpha}}$ and $\underline{\bm{\beta}}$ in the universe of the intended model, the ur-function on $\{\underline{\bm{\alpha}}\}$ that maps $\underline{\bm{\alpha}}$ to $\underline{\bm{\beta}}$ can, using set-builder notation, be identified with the individual $\{\langle \underline{\bm{\alpha}}, \underline{\bm{\beta}}\rangle\}_{\{\underline{\bm{\alpha}}\}}$ in the universe of the intended model. Note that the graph of the ur-function is guaranteed to exist. \hfill$\Box$
\end{Remark}

\begin{Axiom}\label{ax:REGF}\rm Axiom of Regularity for Functions (REG-F): for any set $X$ and any function $f_X$ with any codomain $Y$, $f_X$ does not take itself as argument or has itself as image:\\
\begin{equation}
\forall X \forall f_X\forall Y(f_X: X\twoheadrightarrow Y\Rightarrow \forall \alpha(f_X: f_X \not\mapsto \alpha \wedge f_X:\alpha \not\mapsto f_X))
\end{equation}
$\Box$
\end{Axiom}

\begin{Remark}\rm As to the first part, Wittgenstein already mentioned that a function cannot have itself as argument \cite{Wittgenstein}. The second part is to exclude the existence of pathological `Siamese twin functions', e.g. the ur-function $f_X$ and $g_Y$ given, using set-builder notation, by
\begin{gather}
f_{X}: \{h_{Y}\} \twoheadrightarrow \{f_{X}\}\ , \ f_X: h_Y \mapsto f_X \\
h_{Y}: \{f_{X}\} \twoheadrightarrow \{h_{Y}\}\ , \ g_Y: f_X \mapsto g_Y
\end{gather}
We thus have ${\rm dom}(f_X) = X = \{h_{Y}\}$ and ${\rm dom}(h_Y) = Y = \{f_{X}\}$; if one tries to substitute that in the above Eqs., then one gets `infinite towers'. These may not be constructible from the axioms, but they could exist a priori in the category of sets and functions: to avoid that we practice mathematical eugenics and prevent them from occurring with REG-F. See Fig. \ref{fig:VD} for an illustration. The name `Siamese twin functions' is derived from the name `Siamese twin sets' for sets $A$ and $B$ satisfying $A\in B \wedge B\in A$, as published in \cite{Muller}.\hfill$\Box$
\end{Remark}

\begin{SCfigure}[1.8][h!]
\includegraphics[width=0.45\textwidth]{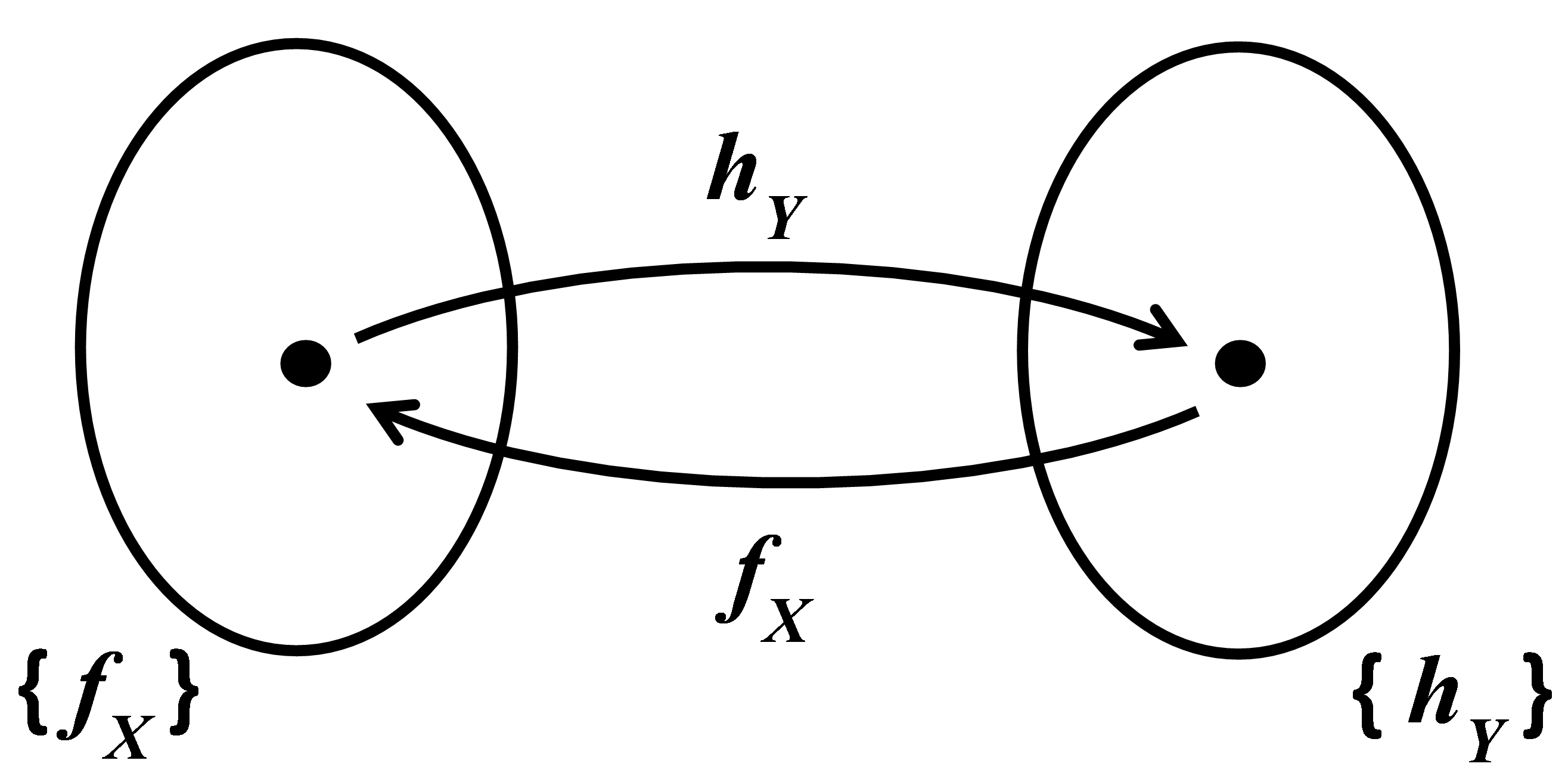}
\caption{Venn diagram of the Siamese twin functions $f_X$ and $h_Y$. The left oval together with the black point inside it is a Venn diagram representing the singleton $\{f_X\}$; the right oval together with the black point inside it is a Venn diagram representing the singleton $\{h_y\}$. The upper arrow represents the mapping of $f_X$ to $h_Y$ by $h_Y$, the lower arrow the mapping of $h_Y$ to $f_X$ by $f_X$.}
\label{fig:VD}
\end{SCfigure}

\subsection{The nonstandard function-theoretical axiom and inference rules}

\begin{Definition}\label{def:extra-vocabulary}\rm
The \textbf{vocabulary} of $L_\mathfrak{T}$ as given by Def. \ref{def:vocabulary} is extended
\begin{enumerate}[label=(\roman*)]
  \item with symbols `$\imath$', the \emph{iota-operator}, and `$\bigwedge$', the \emph{conjunctor};
  \item for any constant ${\rm\hat{\mathbf{X}}}$ denoting a set, with enough composite symbols ${\rm \hat{\mathbf{f}}}_{\alpha^+}$, ${\rm\hat{\mathbf{h}}}_{\beta^+}$, ... such that each of these is a variable that ranges over a family of ur-functions indexed in ${\rm\hat{\mathbf{X}}}$.
\end{enumerate}
The \textbf{syntax} of $L_\mathfrak{T}$ as given by Def. \ref{def:syntax} is extended with the following clauses:
\begin{enumerate}[label=(\roman*),resume]
\item if ${\rm \hat{\mathbf{f}}}_{\alpha^+}$ is a variable as in clause (ii) of Def. \ref{def:extra-vocabulary}, then ${\rm \hat{\mathbf{f}}}_{\alpha^+}$ is a term;
\item if $t$ is a term and $u_{t^+}$ is a composite term with an occurrence of $t$, and $\beta$ is a variable ranging over all things, then $\imath \beta(u_{t^+}:t\mapsto\beta)$ is a iota-term denoting the image of $t$ under the ur-function $u_{t^+}$;
\item if ${\rm\hat{\mathbf{X}}}$ is a constant designating a set, $\alpha$ a simple variable ranging over all things, and $\Psi(\alpha)$ an atomic formula of the type $t:t^\prime\mapsto t^{\prime\prime}$ that is open in $\alpha$, then $\bigwedge_{\alpha\in {\rm\hat{\mathbf{X}}}} \Psi(\alpha)$ is a formula; 
\item if $\Phi$ is a formula with a subformula $\bigwedge_{\alpha\in {\rm\hat{\mathbf{X}}}} \Psi(\alpha)$ as in (iii) with an occurrence of a composite variable $f_{\alpha^+}$, then ${(\forall f_{\alpha^+})}_{\alpha\in {\rm\hat{\mathbf{X}}}}\Phi$ and ${(\exists f_{\alpha^+})}_{\alpha\in {\rm\hat{\mathbf{X}}}}\Phi$ are formulas;
\item if $X$ is a simple variable ranging over sets, and $\Upsilon$ a formula with no occurrence of $X$ but with a subformula ${(\forall f_{\alpha^+})}_{\alpha\in {\rm\hat{\mathbf{X}}}}\Phi$ as in (iv), then $\forall X [X\backslash{\rm\hat{\mathbf{X}}}]\Upsilon$ and $\exists X [X\backslash{\rm\hat{\mathbf{X}}}]\Upsilon$ are formulas with a subformula ${(\forall f_{\alpha^+})}_{\alpha\in X}[X\backslash{\rm\hat{\mathbf{X}}}]\Phi$;
\item if $X$ is a simple variable ranging over sets, and $\Upsilon$ a formula with no occurrence of $X$ but with a subformula ${(\exists f_{\alpha^+})}_{\alpha\in {\rm\hat{\mathbf{X}}}}\Phi$ as in (iv), then $\forall X [X\backslash{\rm\hat{\mathbf{X}}}]\Upsilon$ and $\exists X [X\backslash{\rm\hat{\mathbf{X}}}]\Upsilon$ are formulas with a subformula ${(\exists f_{\alpha^+})}_{\alpha\in X}[X\backslash{\rm\hat{\mathbf{X}}}]\Phi$.
\end{enumerate}
$\Box$
\end{Definition}
\newpage
\begin{Remark}\rm Concerning the iota-operator in clause (iv) of Def. \ref{def:extra-vocabulary}, we thus have
\begin{equation}
  \forall\alpha\forall f_{\alpha^+}\forall \gamma (\gamma = \imath\beta(f_{\alpha^+}:\alpha\mapsto\beta)\Leftrightarrow f_{\alpha^+}:\alpha\mapsto \gamma)
\end{equation}
Note that upon assigning constant values to $\alpha$ and $f_{\alpha^+}$, the term $\imath\beta(f_{\alpha^+}:\alpha\mapsto\beta)$ always refers to an existing, unique thing: there is thus no danger of nonsensical terms involved in this use of the iota-operator.\hfill$\Box$
\end{Remark}
\begin{Definition}\label{def:extra-syntax-elements}\rm
The following special language elements are added:
\begin{enumerate}[label=(\roman*)]
\item if ${\rm\hat{\mathbf{X}}}$ is a constant designating a set, $X$ and $\alpha$ simple variables ranging over sets c.q. things, and $f_{\alpha^+}$ a composite variable ranging over ur-functions on $\alpha^+$, then
    \begin{itemize}
      \item ${(\forall f_{\alpha^+})}_{\alpha\in {\rm\hat{\mathbf{X}}}}$ is a \textbf{multiple universal quantifier};
      \item ${(\exists f_{\alpha^+})}_{\alpha\in {\rm\hat{\mathbf{X}}}}$ is a \textbf{multiple existential quantifier};
      \item ${(\forall f_{\alpha^+})}_{\alpha\in X}$ in the scope of a quantifier $\forall X$ is a \textbf{universally generalized multiple universal quantifier};
      \item ${(\forall f_{\alpha^+})}_{\alpha\in X}$ in the scope of a quantifier $\exists X$ is an \textbf{existentially generalized multiple universal quantifier};
      \item ${(\exists f_{\alpha^+})}_{\alpha\in X}$ in the scope of a quantifier $\forall X$ is a \textbf{universally generalized multiple existential quantifier};
      \item ${(\exists f_{\alpha^+})}_{\alpha\in X}$ in the scope of a quantifier $\exists X$ is an \textbf{existentially generalized multiple existential quantifier};
    \end{itemize}
\item if ${\rm\hat{\mathbf{X}}}$ is a constant designating a set, $X$ a simple variable ranging over sets, and $\alpha$ a simple variable ranging over all things, then
    \begin{itemize}
      \item $\bigwedge_{\alpha\in {\rm\hat{\mathbf{X}}}}$ is a \textbf{conjunctive operator with constant range};
      \item $\bigwedge_{\alpha\in X}$ a \textbf{conjunctive operator with variable range}.
      \end{itemize}
\end{enumerate}
$\Box$
\end{Definition}

\noindent Concerning the language elements in clause (i), if $\rotatechar{K}$ and $\rotatechar{K}^\prime$ are existential or universal quantification symbols, then we can generally say that ${(\rotatechar{K} f_{\alpha^+})}_{\alpha\in {\rm\hat{\mathbf{X}}}}$ is a \textbf{multiple quantifier}, and that ${(\rotatechar{K} f_{\alpha^+})}_{\alpha\in X}$ in the scope of a quantifier $\rotatechar{K}^\prime X$ is a \textbf{generalized multiple quantifier}.

\begin{Definition}\label{def:scope1}\rm If ${\rm\hat{\mathbf{X}}}$ is a constant designating a set and $\bigwedge_{\alpha\in {\rm\hat{\mathbf{X}}}}\Psi$ is a subformula of a formula $\Phi$, then $\Psi$ is the \textbf{scope of the conjunctive operator}; furthermore,
\begin{enumerate}[label=(\roman*)]
\item if there is an occurrence of a variable $\alpha$ and/or ${\rm \hat{\mathbf{f}}}_{\alpha^+}$ in the scope of the conjunctive operator, then a formula $\bigwedge_{\alpha\in {\rm\hat{\mathbf{X}}}}\Psi$ has a \textbf{semantic occurrence} of each of the constants ${\rm\bm{\hat{\alpha}}}$ over which the variable $\alpha$ ranges, and/or of each of the constant ur-functions ${\rm\hat{\mathbf{u}}}_{{\rm\bm{\hat{\alpha}}}^+}$ over which the variable ${\rm \hat{\mathbf{f}}}_{\alpha^+}$ ranges---a subformula $\bigwedge_{\alpha\in {\rm\hat{\mathbf{X}}}}\Psi$ has thus to be viewed as the conjunction of all the formulas $[{{\rm\bm{\hat{\alpha}}}\backslash\alpha][{\rm\hat{\mathbf{u}}}_{{\rm\bm{\hat{\alpha}}}^+}\backslash {\rm \hat{\mathbf{f}}}}_{\alpha^+}]\Psi$ with ${\rm\hat{\bm{\alpha}}}\in {\rm\hat{\mathbf{X}}}$.
\item if there is an occurrence of a composite variable $f_{\alpha^+}$ in the scope of the conjunctive operator, then the subformula $\bigwedge_{\alpha\in {\rm\hat{\mathbf{X}}}}\Psi(f_{\alpha^+})$ has a \textbf{free semantic occurrence} of each of the simple variables $f_{{\rm\bm{\hat{\alpha}}}^+}$ ranging over ur-functions on the singleton of ${\rm\bm{\hat{\alpha}}}$ with ${\rm\bm{\hat{\alpha}}}\in {\rm\hat{\mathbf{X}}}$---the formula $\bigwedge_{\alpha\in {\rm\hat{\mathbf{X}}}}\Psi(f_{\alpha^+})$ has thus to be viewed as the conjunction of all the formulas $[{\rm\bm{\hat{\alpha}}}\backslash\alpha][f_{{\rm\bm{\hat{\alpha}}}^+}\backslash f_{\alpha^+}]\Psi$.
\end{enumerate}
$\Box$
\end{Definition}

\begin{Definition}\label{def:scope2}\rm If ${\rm\hat{\mathbf{X}}}$ is a constant designating a set, $\rotatechar{K}$ an existential or universal quantification symbol, and ${(\rotatechar{K} f_{\alpha^+})}_{\alpha\in {\rm\hat{\mathbf{X}}}} \Psi$ a subformula of a formula $\Phi$, then $\Psi$ is the \textbf{scope of the multiple quantifier}; likewise for the scope of the generalized multiple quantifiers of Def. \ref{def:extra-syntax-elements}. If a formula $\Psi$ has a free semantic occurrence of each of the simple variables $f_{{\rm\bm{\hat{\alpha}}}^+}$ with a constant ${\rm\bm{\hat{\alpha}}}\in {\rm\hat{\mathbf{X}}}$, then a formula ${(\rotatechar{K} f_{\alpha^+})}_{\alpha\in {\rm\hat{\mathbf{X}}}} \Psi$ has a \textbf{bounded semantic occurrence} of each of the simple variables $f_{{\rm\bm{\hat{\alpha}}}^+}$ with a constant ${\rm\bm{\hat{\alpha}}}\in {\rm\hat{\mathbf{X}}}$. A nonstandard formula $\Psi$ without free occurrences of variables is a \textbf{sentence}. If $X$ is a simple variable ranging over sets and $\Psi$ is a sentence with an occurrence of a multiple quantifier ${(\rotatechar{K} f_{\alpha^+})}_{\alpha\in {\rm\hat{\mathbf{X}}}}$ and with no occurrence of $X$, then $\forall X [X\backslash{\rm\hat{\mathbf{X}}}]\Psi$ and $\exists X [X\backslash{\rm\hat{\mathbf{X}}}]\Psi$ are sentences with an occurrence of a generalized multiple quantifier ${(\rotatechar{K} f_{\alpha^+})}_{\alpha\in X}$. \hfill$\Box$
\end{Definition}

\begin{Axiom}\label{ax:SUMF}\rm Sum Function Axiom (SUM-F): for any nonempty set $X$ and for any family of ur-functions $f_{\alpha^+}$ indexed in $X$, there is a sum function $F_X$ with some codomain $Y$ such that the conjunction of all mappings by $F_X$ of $\alpha$ to its image under the ur-function $f_{\alpha^+}$ holds for $\alpha$ ranging over $X$:
\begin{equation}\label{eq:SUMF}
\forall X\neq \emptyset {(\forall f_{\alpha^+})}_{\alpha \in X} \exists F_X \exists Y\left(F_X:X \twoheadrightarrow Y \wedge {\bigwedge}_{\alpha\in X} F_X:\alpha\mapsto \imath\beta(f_{\alpha^+}:\alpha\mapsto\beta)\right)
\end{equation}
$\Box$
\end{Axiom}

\noindent With SUM-F, all non-logical axioms of the present nonstandard theory have been introduced. But still, rules of inference must be given to derive meaningful theorems from SUM-F. So, the rules of inference that follow have to be seen as part of the \emph{logic}.

\begin{Rule}\label{inf:NUE}\rm Nonstandard Universal Elimination:
\begin{equation}
\forall X\Psi\left( {(\rotatechar{K} f_{\alpha^+})}_{\alpha\in X}, {\bigwedge}_{\alpha\in X} \right) \vdash [{\rm\hat{\mathbf{X}}}\backslash X]\Psi  \hfill {\rm for\ any\ constant\ {\rm\hat{\mathbf{X}}}}
\end{equation}
where $\Psi\left( {(\rotatechar{K} f_{\alpha^+})}_{\alpha\in X}, \bigwedge_{\alpha\in X} \right)$ is a formula with an occurrence of a generalized multiple quantifier and of a conjunctive operator with variable range, and where $[{\rm\hat{\mathbf{X}}}\backslash X]\Psi$ is a formula with an occurrence of a multiple quantifier ${(\rotatechar{K} f_{\alpha^+})}_{\alpha\in {\rm\hat{\mathbf{X}}}}$ and of a conjunctive operator $\bigwedge_{\alpha\in {\rm\hat{\mathbf{X}}}}$ with constant range. \hfill$\Box$
\end{Rule}

\noindent Thus speaking, from SUM-F we can deduce a formula
\begin{equation}\label{eq:NUE}
{(\forall f_{\alpha^+})}_{\alpha \in {\rm\hat{\mathbf{X}}}} \exists F_{\rm\hat{\mathbf{X}}} \exists Y\left(F_{\rm\hat{\mathbf{X}}}:{\rm\hat{\mathbf{X}}} \twoheadrightarrow Y \wedge {\bigwedge}_{\alpha\in {\rm\hat{\mathbf{X}}}} F_{\rm\hat{\mathbf{X}}}:\alpha\mapsto \imath\beta(f_{\alpha^+}:\alpha\mapsto\beta)\right)
\end{equation}
for any constant ${\rm\hat{\mathbf{X}}}$ designating a set.

\begin{Rule}\label{inf:MUQE}\rm Multiple Universal Elimination:
\begin{equation}
{(\forall f_{\alpha^+})}_{\alpha \in {\rm\hat{\mathbf{X}}}}\Phi(f_{\alpha^+}) \vdash [{\rm\hat{\mathbf{f}}}_{\alpha^+}\backslash f_{\alpha^+}]\Phi
\end{equation}
where $\Phi(f_{\alpha^+})$ is a formula with an occurrence of the same composite variable $f_{\alpha^+}$ that also occurs in the preceding multiple universal quantifier ${(\forall f_{\alpha^+})}_{\alpha \in {\rm\hat{\mathbf{X}}}}$, and where ${\rm\hat{\mathbf{f}}}_{\alpha^+}$ is a variable as meant in clause (ii) of Def. \ref{def:extra-vocabulary}.\hfill$\Box$
\end{Rule}

\noindent Thus speaking, from a sentence (\ref{eq:NUE}), which is an instance of SUM-F derived by inference rule \ref{inf:NUE}, we can derive a formula
\begin{equation}\label{eq:MUQE}
\exists F_{\rm\hat{\mathbf{X}}} \exists Y\left(F_{\rm\hat{\mathbf{X}}}:{\rm\hat{\mathbf{X}}} \twoheadrightarrow Y \wedge {\bigwedge}_{\alpha\in {\rm\hat{\mathbf{X}}}} F_{\rm\hat{\mathbf{X}}}:\alpha\mapsto \imath\beta({\rm\hat{\mathbf{f}}}_{\alpha^+}:\alpha\mapsto\beta)\right)
\end{equation}
for each variable ${\rm\hat{\mathbf{f}}}_{\alpha^+}$ ranging over a family of ur-functions indexed in ${\rm\hat{\mathbf{X}}}$. Note that the range of such a variable  ${\rm\hat{\mathbf{f}}}_{\alpha^+}$ is constructed by assigning to each of the simple variables $f_{{\rm\bm{\hat{\alpha}}}^+}$ semantically occurring in formula (\ref{eq:NUE}) a constant value ${\rm\hat{\mathbf{u}}}_{{\rm\bm{\hat{\alpha}}}^+}$.

\begin{Rule}\label{inf:RuleC}\rm Nonstandard Rule-C:
\begin{equation}
\exists t \Phi \vdash [{\rm\hat{\mathbf{t}}}\backslash t]\Phi
\end{equation}
where $t$ is a simple variable $x$ ranging over sets or a simple variable $f_{\rm\hat{\mathbf{X}}}$ ranging over functions on a constant set, and where ${\rm\hat{\mathbf{t}}}$ is a constant in the range of $t$ that does not occur in $\Phi$ but for which $[{\rm\hat{\mathbf{t}}}\backslash t]\Phi$ holds. If $\Phi$ has an occurrence of a generalized multiple quantifier ${(\rotatechar{K} f_{\alpha^+})}_{\alpha\in t}$, then $[{\rm\hat{\mathbf{t}}}\backslash t]\Phi$ has an occurrence of a multiple quantifier ${(\rotatechar{K} f_{\alpha^+})}_{\alpha\in {\rm\hat{\mathbf{t}}}}$; if $\Phi$ has an occurrence of a conjunctive operator $\bigwedge_{\alpha\in t}$ with variable range, then $[{\rm\hat{\mathbf{t}}}\backslash t]\Phi$ has an occurrence of a conjunctive operator $\bigwedge_{\alpha\in {\rm\hat{\mathbf{t}}}}$ with constant range. \hfill$\Box$
\end{Rule}

\noindent Thus speaking, from SUM-F we can deduce a formula
\begin{equation}\label{eq:RuleC}
\exists Y({\rm\hat{\mathbf{F}}}_{\rm\hat{\mathbf{X}}}:{\rm\hat{\mathbf{X}}}\twoheadrightarrow Y) \wedge
{\bigwedge}_{\alpha\in {\rm\hat{\mathbf{X}}}} {\rm\hat{\mathbf{F}}}_{\rm\hat{\mathbf{X}}}:\alpha\mapsto \imath\beta({\rm\hat{\mathbf{f}}}_{\alpha^+}:\alpha\mapsto\beta)
\end{equation}
which is a conjunction of a standard first-order formula and a nonstandard formula with an occurrence of the new constant ${\rm\hat{\mathbf{F}}}_{\rm\hat{\mathbf{X}}}$, designating the sum function on ${\rm\hat{\mathbf{X}}}$, in the scope of a conjunctive operator. Of course, this conjunction $\Psi \wedge \Phi$ is true if and only if both its members are true. This requires one more inference rule.

\begin{Rule}\label{inf:CES}\rm Conjunctive Operator Elimination:
\begin{equation}
{\bigwedge}_{\alpha\in {\rm\hat{\mathbf{X}}}} \Psi(\alpha) \vdash [{\rm\bm{\hat{\alpha}}}\backslash \alpha] \Psi(\alpha)
\end{equation}
where $\Psi(\alpha)$ is a formula of the type $t:t^\prime \mapsto t^{\prime\prime}$ that is open in $\alpha$, and ${\rm\bm{\hat{\alpha}}}$ any constant designating an element of ${\rm\hat{\mathbf{X}}}$. \hfill$\Box$
\end{Rule}

\noindent Thus speaking, from the right member of the conjunction (\ref{eq:RuleC}) we can derive an entire scheme, consisting of one standard first-order formula
\begin{equation}\label{eq:CES}
{\rm\hat{\mathbf{F}}}_{\rm\hat{\mathbf{X}}}:{\rm\bm{\hat{\alpha}}}\mapsto \imath\beta({\rm\hat{\mathbf{u}}}_{{\rm\bm{\hat{\alpha}}}^+}:{\rm\bm{\hat{\alpha}}}\mapsto\beta)
\end{equation}
for each constant ${\rm\hat{\mathbf{u}}}_{{\rm\bm{\hat{\alpha}}}^+}$. So given an infinitary conjunction $\bigwedge_{\alpha\in {\rm\hat{\mathbf{X}}}} {\rm\hat{\mathbf{F}}}_{\rm\hat{\mathbf{X}}}:\alpha\mapsto \imath\beta({\rm \hat{\mathbf{f}}}_{\alpha^+}:\alpha\mapsto\beta)$, the sentences (\ref{eq:CES}) derived by rule \ref{inf:CES} are true for each constant ${\rm\hat{\mathbf{u}}}_{{\rm\bm{\hat{\alpha}}}^+}$ semantically occurring in Eq. (\ref{eq:RuleC}).

\begin{Remark}\label{rem:ModelOfFunction}\rm
Given an individual set ${\rm\underline{\mathbf{X}}}$ and a variable ${\rm\underline{\mathbf{f}}}_{\alpha^+}$ that ranges over a family of ur-functions indexed in ${\rm\underline{\mathbf{X}}}$ in the universe of the intended model, the unique sum function ${\rm\underline{\mathbf{F}}}_{\rm\underline{\mathbf{X}}}$ for which
\begin{equation}
{\bigwedge}_{\alpha\in {\rm\underline{\mathbf{X}}}} {\rm\underline{\mathbf{F}}}_{\rm\underline{\mathbf{X}}}:\alpha\mapsto \imath\beta({\rm\underline{\mathbf{f}}}_{\alpha^+}:\alpha\mapsto\beta)
\end{equation}
can, using set-builder notation, be identified with the individual
\begin{equation}
  {\rm\underline{\mathbf{F}}}_{\rm\underline{\mathbf{X}}} = \{\langle \alpha,\beta\rangle \ | \ \alpha\in{\rm\underline{\mathbf{X}}}\wedge\beta={\rm\underline{\mathbf{f}}}_{\alpha^+}(\alpha)\}_{\rm\underline{\mathbf{X}}}
\end{equation}
in the universe of the intended model. The graph of ${\rm\underline{\mathbf{F}}}_{\rm\underline{\mathbf{X}}}$, i.e. the set $\{\langle \alpha,\beta\rangle \ | \ \alpha\in{\rm\underline{\mathbf{X}}}\wedge\beta={\rm\underline{\mathbf{f}}}_\alpha(\alpha)\}$ is certain to exist, see Th. \ref{th:graph} (next section). So constructing a sum function is a means to constructing a set. \hfill$\Box$
\end{Remark}

\begin{Example}\label{ex:1-omega}\rm Consider the infinite ordinal $\bm{\omega}$ from Ax. \ref{ax:INF}: its elements are the finite ordinals 0, 1, 2, ... Applying Nonstandard Universal Elimination, we thus deduce from SUM-F that
\begin{equation}\label{eq:GMQEforB}
{(\forall f_{\alpha^+})}_{\alpha \in \bm{\omega}} \exists F_{\bm{\omega}} \exists Y
\left( F_{\bm{\omega}}:\bm{\omega} \twoheadrightarrow Y \wedge {\bigwedge}_{\alpha\in \bm{\omega}} F_{\mathbb{B}}:\alpha\mapsto \imath\beta(f_{\{\alpha\}}:\alpha\mapsto\beta)\right)
\end{equation}
On account of the ur-function axiom \ref{ax:UFA} we have
\begin{equation}
  \forall x \in \bm{\omega}\exists f_{x^+}(f_{x^+}: x\mapsto x)
\end{equation}
That is, for any finite ordinal $x$ there is an ur-function that on the singleton of $x$ that maps $x$ to itself. Let the variable ${\rm\hat{\mathbf{f}}}^{1}_{\alpha^+}$ range over these identity ur-functions; applying Multiple Universal Elimination to the sentence (\ref{eq:GMQEforB}) then yields a sentence
\begin{equation}\label{eq:existsFonB}
\exists F_{\bm{\omega}} \exists Y\left(F_{\bm{\omega}}: \bm{\omega} \twoheadrightarrow Y \wedge {\bigwedge}_{\alpha\in \bm{\omega}} F_{\bm{\omega}}:\alpha\mapsto \imath\beta({\rm\hat{\mathbf{f}}}^{1}_{\alpha^+}:\alpha\mapsto\beta)\right)
\end{equation}
Introducing the new constant $1_{\bm{\omega}}$ by applying Rule-C to the sentence (\ref{eq:existsFonB}) and substituting $\imath\beta({\rm\hat{\mathbf{f}}}^{1}_{\alpha^+}:\alpha\mapsto\beta) = \alpha$ then yields the conjunction
\begin{equation}\label{eq:F-B}
1_{\bm{\omega}}:\bm{\omega} \twoheadrightarrow \bm{\omega} \wedge {\bigwedge}_{\alpha\in \bm{\omega}} 1_{\bm{\omega}}:\alpha\mapsto \alpha)
\end{equation}
By applying Conjunctive Operator Elimination to the right member of this conjunction (\ref{eq:F-B}), we obtain the countable scheme
\begin{equation}
\left\{ \begin{array}{l}
          1_{\bm{\omega}}:0 \mapsto 0 \\
          1_{\bm{\omega}}:1 \mapsto 1 \\
          1_{\bm{\omega}}:2 \mapsto 2 \\
          \ \ \vdots
        \end{array}
\right.
\end{equation}
This example demonstrates, strictly within the language of $\mathfrak{T}$, how SUM-F and the inference rules can be used to construct the identity function on $\bm{\omega}$ from a family of ur-functions indexed in $\bm{\omega}$. \hfill$\Box$
\end{Example}

\begin{Remark}\label{rem:sentences}\rm Summarizing, it has thus to be taken
\begin{enumerate}[label=(\roman*)]
\item that SUM-F is a typographically finite sentence;
\item that an instance (\ref{eq:NUE}) of SUM-F, deduced by applying Nonstandard Universal Elimination, is a typographically finite sentence;
\item that a formula (\ref{eq:MUQE}), deduced from an instance of SUM-F by applying Multiple Universal Elimination, is a typographically finite sentence;
\item that a conjunction (\ref{eq:RuleC}), deduced by applying Rule-C to a sentence deduced from SUM-F by successively applying Nonstandard Universal Elimination and Multiple Universal Elimination, is a typographically finite sentence.
\end{enumerate}
We thus get that SUM-F being true means that every instance (\ref{eq:NUE}) of SUM-F obtained by Nonstandard Universal Elimination is true; that an instance (\ref{eq:NUE}) of SUM-F being true means that for any variable ${\rm\hat{\mathbf{f}}}_{\alpha^+}$ ranging over a family of ur-functions indexed in ${\rm\hat{\mathbf{X}}}$, formula (\ref{eq:MUQE}) is true; and that a nonstandard formula (\ref{eq:MUQE}) with an occurrence of a variable ${\rm\hat{\mathbf{f}}}_{\alpha^+}$ being true means that, after applying Rule-C, the scheme of standard formulas (\ref{eq:CES}) obtained by Conjunctive Operator Elimination is true---one true standard formula obtains for every ur-function ${\rm\hat{\mathbf{u}}}_{{\rm\bm{\hat{\alpha}}}^+}$ in the range of the variable ${\rm\hat{\mathbf{f}}}_{\alpha^+}$.\hfill$\Box$
\end{Remark}


\noindent This concludes the axiomatic introduction of the nonstandard theory $\mathfrak{T}$. Since we are primarily interested in the theorems that can be derived from the axioms of $\mathfrak{T}$, no rules have been given for the introduction of (multiple) quantifiers or conjunctive operators. Below such rules are given for the sake of completeness, but these will not be discussed.
\begin{Rule}\rm Conjunctive Operator Introduction:
\begin{equation}\label{eq:COE}
 {\{ [I(\alpha)\backslash\alpha]\Psi(\alpha) \}}_{I(\alpha)\in \mathrm{\hat{\mathbf{X}}}}\vdash
{\bigwedge}_{\alpha\in {\rm\hat{\mathbf{X}}}} \Psi(\alpha)
\end{equation}
where $\Psi(\alpha)$ is a formula of the type $t:t^\prime\mapsto t^{\prime\prime}$ that is open in $\alpha$, and ${\{ [I(\alpha)\backslash\alpha]\Psi(\alpha) \}}_{I(\alpha)\in \mathrm{\hat{\mathbf{X}}}}$ is a possibly infinite collection of formulas, each of which is obtained by interpreting the variable $\alpha$ as a constant $I(\alpha) \in \mathrm{\hat{\mathbf{X}}}$ and replacing $\alpha$ in $\Psi(\alpha)$ everywhere by $I(\alpha)$.\hfill$\Box$
\end{Rule}

\noindent Note that the \textbf{collection} of formulas ${\{ [I(\alpha)\backslash\alpha]\Psi(\alpha) \}}_{I(\alpha)\in \mathrm{\hat{\mathbf{X}}}}$ in Eq. (\ref{eq:COE}) is itself not a well-formed formula of the language $L_{\mathfrak{T}}$, but each of the formulas in the collection is.

\begin{Rule}\rm Multiple Universal Introduction:
\begin{equation}
\Phi\left({{\bigwedge}_{\alpha\in {\rm\hat{\mathbf{X}}}}\Psi({\rm\hat{\mathbf{f}}}_{\alpha^+}) }\right)
\vdash
{(\forall f_{\alpha^+})}_{\alpha \in {\rm\hat{\mathbf{X}}}}[f_{\alpha^+}\backslash{\rm\hat{\mathbf{f}}}_{\alpha^+}]\Phi
\end{equation}
where $\Phi\left({{\bigwedge}_{\alpha\in {\rm\hat{\mathbf{X}}}}\Psi({\rm\hat{\mathbf{f}}}_{\alpha^+}) }\right)$ denotes a formula $\Phi$ with a subformula ${\bigwedge}_{\alpha\in {\rm\hat{\mathbf{X}}}}\Psi({\rm\hat{\mathbf{f}}}_{\alpha^+})$ (implying that $\Psi$ is an atomic formula of the type $t:t^\prime\mapsto t^{\prime\prime}$), and where the variable ${\rm\hat{\mathbf{f}}}_{\alpha^+}$ ranges over an \textbf{arbitrary} family of ur-functions indexed in ${\rm\hat{\mathbf{X}}}$.\hfill$\Box$
\end{Rule}

\begin{Rule}\rm Multiple Existential Introduction:
\begin{equation}
\Phi\left({{\bigwedge}_{\alpha\in {\rm\hat{\mathbf{X}}}}\Psi({\rm\hat{\mathbf{f}}}_{\alpha^+}) }\right)
\vdash
{(\exists f_{\alpha^+})}_{\alpha \in {\rm\hat{\mathbf{X}}}}[f_{\alpha^+}\backslash{\rm\hat{\mathbf{f}}}_{\alpha^+}]\Phi
\end{equation}
where $\Phi\left({{\bigwedge}_{\alpha\in {\rm\hat{\mathbf{X}}}}\Psi({\rm\hat{\mathbf{f}}}_{\alpha^+}) }\right)$ denotes a formula $\Phi$ with a subformula ${\bigwedge}_{\alpha\in {\rm\hat{\mathbf{X}}}}\Psi({\rm\hat{\mathbf{f}}}_{\alpha^+})$ (implying that $\Psi$ is an atomic formula of the type $t:t^\prime\mapsto t^{\prime\prime}$), and where the variable ${\rm\hat{\mathbf{f}}}_{\alpha^+}$ ranges over a \textbf{specific} family of ur-functions indexed in ${\rm\hat{\mathbf{X}}}$.\hfill$\Box$
\end{Rule}

\begin{Remark}\rm
Nonstandard Universal Quantification, i.e. the rule
\begin{equation}
 \Psi({\rm\hat{\mathbf{X}}}) \vdash \forall X [X\backslash {\rm\hat{\mathbf{X}}}]\Psi
\end{equation}
for a nonstandard formula $\Psi$ with an occurrence of an \textbf{arbitrary} constant ${\rm\hat{\mathbf{X}}}$, and Nonstandard Existential Quantification, i.e. the rule
\begin{equation}
\Psi({\rm\hat{\mathbf{X}}}) \vdash \exists X [X\backslash {\rm\hat{\mathbf{X}}}]\Psi
\end{equation}
for a nonstandard formula $\Psi$ with an occurrence of a \textbf{specific} constant ${\rm\hat{\mathbf{X}}}$, are the same as in the standard case, but with the understanding that upon quantification a multiple quantifier ${(\rotatechar{K} f_{\alpha^+})}_{\alpha\in {\rm\hat{\mathbf{X}}}}$ in $\Psi$ becomes a generalized multiple quantifier ${(\rotatechar{K} f_{\alpha^+})}_{\alpha\in X}$ in $[X\backslash {\rm\hat{\mathbf{X}}}]\Psi$, and a conjunctive operator $\bigwedge_{\alpha\in {\rm\hat{\mathbf{X}}}}$ with constant range in $\Psi$ becomes a conjunctive operator $\bigwedge_{\alpha\in X}$ with variable range in $[X\backslash {\rm\hat{\mathbf{X}}}]\Psi$. \hfill$\Box$
\end{Remark}

\section{Discussion}

\subsection{Main theorems}

\begin{Theorem}\label{th:graph}\rm Graph Theorem: for any set $X$ and any function $f_X$ with any codomain $Y$, there is a set $Z$ that is precisely the graph of the function $f_X$---that is, there is a set $Z$ whose elements are precisely the two-tuples $\langle\alpha,\beta\rangle$ made up of arguments and images of the function $f_X$. In a formula:
\begin{equation}\label{eq:graph}
\forall X \forall f_X \forall Y \left(f_X:X\twoheadrightarrow Y \Rightarrow \exists Z\forall \zeta (\zeta \in Z \Leftrightarrow \exists \alpha\exists\beta (\zeta = \langle\alpha,\beta\rangle \wedge f_X:\alpha \mapsto\beta)))\right)
\end{equation}
$\Box$
\end{Theorem}
\begin{proof} Let ${\rm\hat{\mathbf{X}}}$ be an arbitrary set, and let the function ${\rm\hat{\mathbf{f}}}_{\rm\hat{\mathbf{X}}}$ be an arbitrary function on ${\rm\hat{\mathbf{X}}}$.  On account of GEN-F (Ax. \ref{ax:GEN-F}), for any $\alpha\in X$ there is then precisely one $\beta$ such that ${\rm\hat{\mathbf{f}}}_{\rm\hat{\mathbf{X}}}:\alpha\mapsto\beta$ . Using Def. \ref{def:TwoTuple}, there exists then for each $\alpha\in {\rm\hat{\mathbf{X}}}$ a singleton $\langle\alpha,\beta\rangle^+$ such that ${\rm\hat{\mathbf{f}}}_{\rm\hat{\mathbf{X}}}:\alpha\mapsto\beta$. On account of the ur-function axiom (Ax. \ref{ax:UFA}), there exists then also an ur-function $u_{\alpha^+}: \alpha^+ \twoheadrightarrow \langle\alpha,\beta\rangle^+\ , \ u_{\alpha^+}: \alpha \mapsto \langle\alpha,\beta\rangle$ for each $\alpha\in {\rm\hat{\mathbf{X}}}$. Thus, on account of SUM-F there is a sum function ${\rm\hat{\mathbf{G}}}_{\rm\hat{\mathbf{X}}}$ with some codomain $Z$ such that ${\rm\hat{\mathbf{G}}}_{\rm\hat{\mathbf{X}}}$ maps every $\alpha\in X$ precisely to the two-tuple $\langle\alpha, \beta\rangle$ for which ${\rm\hat{\mathbf{f}}}_{\rm\hat{\mathbf{X}}}:\alpha\mapsto\beta$. On account of GEN-F, the codomain $Z$ of ${\rm\hat{\mathbf{G}}}_{\rm\hat{\mathbf{X}}}$ exists, and on account of Ax. \ref{ax:NEGC} it is unique: this codomain is precisely the graph of ${\rm\hat{\mathbf{f}}}_{\rm\hat{\mathbf{X}}}$. Since ${\rm\hat{\mathbf{X}}}$ and ${\rm\hat{\mathbf{f}}}_{\rm\hat{\mathbf{X}}}$ were arbitrary, the Graph Theorem follows from universal generalization.
\end{proof}

\noindent In the intended model of $\mathfrak{T}$, there is thus no risk involved in identifying a function $f$ with the individual ${\rm graph}(f)_{{\rm dom}(f)}$, where ${\rm graph}(f)$ is the graph of $f$ and ${\rm dom}(f)$ the domain of $f$, cf. Rem. \ref{rem:ModelOfFunction}.

\begin{Theorem}\label{th:MAIN}\rm Main Theorem: for any nonempty set $X$, if there is a functional relation $\Phi(\alpha, \beta)$ that relates every $\alpha$ in $X$ to precisely one $\beta$, then there is a function $F_X$ with some codomain $Y$ that maps every $\eta \in X$ to precisely that $\xi \in Y$ for which $\Phi(\eta, \xi)$. In a formula, using the iota-operator:
\begin{equation}\label{eq:main}
\forall X\neq\emptyset (\forall \alpha \in X \exists! \beta \Phi(\alpha, \beta) \Rightarrow \exists F_X\exists Y (F_X:X\twoheadrightarrow Y \wedge \forall \eta \in X(F_X: \eta \mapsto \imath\xi\Phi(\eta, \xi))))
\end{equation}
$\Box$
\end{Theorem}
\begin{proof} Let ${\rm\hat{\mathbf{X}}}$ be an arbitrary nonempty set. Suppose then, that for every $\alpha \in \mathbf{X}$ we have precisely one $\beta$ such that $\Phi(\alpha, \beta)$. On account of the ur-function axiom (Ax. \ref{ax:UFA}), for an arbitrary constant ${\rm\bm{\hat{\alpha}}} \in {\rm\hat{\mathbf{X}}}$ there exists then also an ur-function ${\rm\hat{\mathbf{u}}}_{{\rm\bm{\hat{\alpha}}}^+}$ for which
\begin{equation}\label{eq:ur-implicit}
{\rm\hat{\mathbf{u}}}_{\rm\bm{\hat{\alpha}}^+}:{\rm\bm{\hat{\alpha}}}\mapsto \imath\beta\Phi({\rm\bm{\hat{\alpha}}},\beta)
\end{equation}
Let the variable ${\rm\hat{\mathbf{f}}}_{\alpha^+}$ range over these ur-functions. We then deduce from SUM-F by applying Nonstandard Universal Elimination and subsequently Multiple Universal Elimination that
\begin{equation}
\exists F_{\rm\hat{\mathbf{X}}}\exists Y \left(F_{\rm\hat{\mathbf{X}}}:{\rm\hat{\mathbf{X}}}\twoheadrightarrow Y \wedge {\bigwedge}_{\alpha \in {\rm\hat{\mathbf{X}}}}F_{\rm\hat{\mathbf{X}}}: \alpha \mapsto \imath\beta({\rm\hat{\mathbf{f}}}_{\alpha^+}: \alpha \mapsto \beta)\right)
\end{equation}
By subsequently applying Rule-C and Conjunctive Operator Elimination we then deduce the scheme
\begin{equation}
{\rm\hat{\mathbf{F}}}_{\rm\hat{\mathbf{X}}}:{\rm\bm{\hat{\alpha}}}\mapsto \imath\beta\Phi({\rm\bm{\hat{\alpha}}},\beta)
\end{equation}
for the sum function ${\rm\hat{\mathbf{F}}}_{\rm\hat{\mathbf{X}}}$. Generalizing this scheme we obtain
\begin{equation}
\forall \eta \in {\rm\hat{\mathbf{X}}}({\rm\hat{\mathbf{F}}}_{\rm\hat{\mathbf{X}}}: \eta \mapsto \imath\xi\Phi(\eta, \xi)))
\end{equation}
We thus obtain
\begin{equation}\label{eq:proofmain}
\exists F_{\rm\hat{\mathbf{X}}}\exists Y \left(F_{\rm\hat{\mathbf{X}}}:{\rm\hat{\mathbf{X}}}\twoheadrightarrow Y \wedge
\forall \eta \in {\rm\hat{\mathbf{X}}}(F_{\rm\hat{\mathbf{X}}}: \eta \mapsto \imath\xi\Phi(\eta, \xi)))\right)
\end{equation}
Since the functional relation was assumed, we get $\forall\alpha \in {\rm\hat{\mathbf{X}}} \exists!\beta \Phi(\alpha, \beta) \Rightarrow \Psi$ where $\Psi$ is formula (\ref{eq:proofmain}). Since ${\rm\hat{\mathbf{X}}}$ was an arbitrary nonempty set, we can quantify over nonempty sets. This gives precisely the requested formula (\ref{eq:main}).
\end{proof}
\begin{Remark}\rm Theorem \ref{th:MAIN} is an infinite scheme, with one formula for every functional relation $\Phi$. The point is this: given a set $X$, on account of this theorem we can \textbf{construct} a function $f_X$ by giving a function prescription---what we then actually do is defining an ur-function for every $\alpha\in X$; the function $f_X$ then exists on account of SUM-F. And by constructing the function we construct its graph, which exists on account of Th. \ref{th:graph}. Generally speaking, if we define an ur-function for each singleton $\alpha^+\subset X$, then we do not yet have the graphs of these ur-functions in a set. But in the present framework, the set of these graphs is guaranteed to exist. Ergo, giving a function prescription \emph{is} constructing a set!\hfill$\Box$
\end{Remark}

\subsection{Derivation of SEP and SUB of ZF}
We start by proving that the infinite axiom scheme SEP of ZF is a theorem (scheme) of our theory $\mathfrak{T}$:

\begin{Theorem}\label{SEP}\rm Separation Axiom Scheme of ZF:\\
$\forall X \exists Y \forall \alpha( \alpha \in Y \Leftrightarrow \alpha \in X \wedge \Phi(\alpha))$
\end{Theorem}
\begin{proof} Let ${\rm\hat{\mathbf{X}}}$ be an arbitrary set and let $\Phi$ be an arbitrary unary relation on ${\rm\hat{\mathbf{X}}}$. On account of the ur-function axiom (Ax. \ref{ax:UFA}), for an arbitrary constant ${\rm\bm{\hat{\alpha}}} \in {\rm\hat{\mathbf{X}}}$ there exists then an ur-function ${\rm\hat{\mathbf{u}}}_{\rm\bm{\hat{\alpha}}^+}$ for which
\begin{equation}
\left\{ \begin{array}{l} {\rm\hat{\mathbf{u}}}_{\rm\bm{\hat{\alpha}}^+}: {\rm\bm{\hat{\alpha}}} \mapsto 1  \ \ if \ \Phi({\rm\bm{\hat{\alpha}}}) \\
{\rm\hat{\mathbf{u}}}_{\rm\bm{\hat{\alpha}}^+}: {\rm\bm{\hat{\alpha}}} \mapsto 0 \ \ if \ \neg \Phi({\rm\bm{\hat{\alpha}}}) \end{array} \right .
\end{equation}
Let the variable ${\rm\hat{\mathbf{f}}}_{\eta^+}$ range over these ur-functions. On account of Th. \ref{th:MAIN} we then get
\begin{equation}
  \exists F_{\rm\hat{\mathbf{X}}}\forall \eta \in {\rm\hat{\mathbf{X}}}(F_{\rm\hat{\mathbf{X}}}: \eta \mapsto \imath\beta({\rm\hat{\mathbf{f}}}_{\eta^+}:\eta\mapsto\beta))
\end{equation}
Let this sum function be designated by the constant ${\rm\hat{\mathbf{F}}}_{\rm\hat{\mathbf{X}}}$. On account of INV (Ax. \ref{ax:INV}), the inverse image set ${\rm\hat{\mathbf{F}}}_{\rm\hat{\mathbf{X}}}^{-1}(1)$ exists: we then have $\forall \alpha( \alpha \in {\rm\hat{\mathbf{F}}}_{\rm\hat{\mathbf{X}}}^{-1}(1) \Leftrightarrow \alpha \in {\rm\hat{\mathbf{X}}} \wedge \Phi(\alpha))$. Th. \ref{SEP} then obtains from here by existential generalization and universal generalization.
\end{proof}
\noindent Proceeding, we prove that the infinite axiom scheme SUB of ZF is a theorem (scheme) of our theory $\mathfrak{T}$:
\begin{Theorem}\label{SUB}\rm Substitution Axiom Scheme of ZF:\\
$\forall X(\forall \alpha \in X \exists! \beta \Phi(\alpha, \beta) \Rightarrow \exists Z \forall \gamma( \gamma \in Z \Leftrightarrow \exists \xi (\xi \in X \wedge \Phi(\xi, \gamma))))$
\end{Theorem}
\begin{proof}
Let ${\rm\hat{\mathbf{X}}}$ be an arbitrary set and let there for every $\alpha \in {\rm\hat{\mathbf{X}}}$ be precisely one $\beta$ such that $\Phi(\alpha, \beta)$. Then on account of Th. \ref{th:MAIN}, a sum function ${\rm\hat{\mathbf{F}}}_{\rm\hat{\mathbf{X}}}$ exists for which
\begin{equation}
\forall \alpha \in X ({\rm\hat{\mathbf{F}}}_{\rm\hat{\mathbf{X}}}:{\rm\bm{\hat{\alpha}}}\mapsto \imath\beta\Phi({\rm\bm{\hat{\alpha}}},\beta))
\end{equation}
On account of Ax. \ref{ax:NEGC}, the codomain of ${\rm\hat{\mathbf{F}}}_{\rm\hat{\mathbf{X}}}$ is the image set; denoting this by ${\rm\hat{\mathbf{F}}}_{\rm\hat{\mathbf{X}}}[{\rm\hat{\mathbf{X}}}]$ we then have
\begin{equation}
\forall \gamma( \gamma \in {\rm\hat{\mathbf{F}}}_{\rm\hat{\mathbf{X}}}[{\rm\hat{\mathbf{X}}}] \Leftrightarrow \exists \xi (\xi \in X \wedge \Phi(\xi, \gamma)))
\end{equation}
Since the functional relation $\Phi$ on the arbitrary set ${\rm\hat{\mathbf{X}}}$ was assumed, this is implied by $\forall \alpha \in {\rm\hat{\mathbf{X}}} \exists! \beta \Phi(\alpha, \beta)$. We write out this implication: Th. \ref{SUB} then obtains by existential generalization and universal generalization.
\end{proof}
\noindent These two theorems schemes provide an argument for considering our theory $\mathfrak{T}$ to be not weaker than ZF. However, because in the framework of ZF elements of sets are always sets while in the present framework elements of sets may also be functions (which are not sets), we have strictly speaking not yet proven that every result in ZF about sets automatically translates to the present framework.

\begin{Remark}\rm Should further research on $\mathfrak{T}$ reveal unintended consequences that render it inconsistent or otherwise useless, there is still the possibility to remove SUM-F from $\mathfrak{T}$ add the above theorem schemes \ref{SEP} and \ref{SUB} as axioms to $\mathfrak{T}$. That still gives a theory---although a standard one with infinitely many axioms---that merges set theory and category theory into a single framework.\hfill$\Box$

\end{Remark}

\subsection{Model theory}

\begin{Definition}\label{def:model}\rm
A \textbf{model} $\mathcal{M}$ of the present theory $\mathfrak{T}$ consists of the \textbf{universe} $|\mathcal{M}|$ of $\mathcal{M}$, which is a concrete category made up of a nonempty collection of objects (sets) and a nonempty collection of arrows (functions on sets), and the \textbf{language} $L_\mathcal{M}$ of $\mathcal{M}$, which is the language $L_\mathfrak{T}$ of $\mathfrak{T}$ extended with a constant for every object and for every arrow in $|\mathcal{M}|$, such that the axioms of $\mathfrak{T}$ are valid in $\mathcal{M}$.\hfill$\Box$
\end{Definition}

\noindent In standard first-order logic it is well defined what it means that a formula is `valid' in a model $M$. This notion of validity translates to the framework of $\mathfrak{T}$ for all standard formulas. However, it remains to be established what it means that SUM-F and nonstandard consequences thereof are valid in a model $\mathcal{M}$ of $\mathfrak{T}$. Recall that symbols referring to individuals in $|\mathcal{M}|$ will be underlined to distinguish these individuals from constants of $\mathfrak{T}$.
\begin{Definition}\label{def:validity}\rm (Validity of nonstandard formulas.)
\begin{enumerate}[label=(\roman*)]
  \item a sentence $\forall X \neq \emptyset \Psi$ with a nonstandard subformula $\Psi$, such as the sum function axiom, is \textbf{valid} in a model $\mathcal{M}$ of $\mathfrak{T}$ if and only if for every assignment $g$ that assigns an individual nonempty set $g(X) = \mathrm{\underline{\mathbf{X}}}$ in $|\mathcal{M}|$ as a value to the variable $X$, $[\mathrm{\underline{\mathbf{X}}}\backslash X]\Psi$ is valid in $\mathcal{M}$;
  \item a sentence ${(\forall f_{\alpha^+})}_{\alpha \in {\rm\underline{\mathbf{X}}}}\Phi$ with an occurrence of an individual nonempty set ${\rm\underline{\mathbf{X}}}$ of $|\mathcal{M}|$, such as an instance of SUM-F, is \textbf{valid} in a model $\mathcal{M}$ of $\mathfrak{T}$ if and only if for every `team assignment' $g$ that assigns an individual ur-function $g(f_{\underline{\bm{\alpha}}^+})= \mathrm{\underline{\mathbf{u}}}_{\underline{\bm{\alpha}}^+}$ in $|\mathcal{M}|$ as a value to each variable $f_{\underline{\bm{\alpha}}^+}$ semantically occurring in $\Phi$, the sentence $[\mathrm{\underline{\mathbf{f}}}^g_{\alpha^+} \backslash f_{\alpha^+}]\Phi$ with the variable $\mathrm{\underline{\mathbf{f}}}^g_{\alpha^+}$ ranging over the family of ur-functions ${\left( \mathrm{\underline{\mathbf{u}}}_{\alpha^+}\right)}_{\alpha\in {\rm\underline{\mathbf{X}}}}$ is valid in $\mathcal{M}$;
  \item a sentence $\exists t \Upsilon$ with an occurrence of a simple variable $t$ ranging over sets or over functions on a set $\mathrm{\underline{\mathbf{X}}}$ and with $\Upsilon$ being a nonstandard formula, such as the sentences that can be obtained by successively applying Nonstandard Universal Elimination and Multiple Universal Elimination to SUM-F, is \textbf{valid} in a model $\mathcal{M}$ of $\mathfrak{T}$ if and only if for at least one assignment $g$ that assigns an individual function $g(t) = \mathrm{\underline{\mathbf{F}}}_\mathrm{\underline{\mathbf{X}}}$ or an individual nonempty set $g(t)=\mathrm{\underline{\mathbf{Y}}}$ as value to the variable $t$, the sentence $[g(t)\backslash t]\Upsilon$ is valid in $\mathcal{M}$;
  \item A sentence ${\bigwedge}_{\alpha\in \mathrm{\underline{\mathbf{X}}}} \Psi(\mathrm{\underline{\mathbf{f}}}_{\alpha^+},\alpha)$ is \textbf{valid} in a model $\mathcal{M}$ of $\mathfrak{T}$ if and only if for every assignment $g$ that assigns an individual ur-function $g(\mathrm{\underline{\mathbf{f}}}_{\alpha^+})= \mathrm{\underline{\mathbf{u}}}_{\underline{\bm{\alpha}}^+}$ from the range ${\left( \mathrm{\underline{\mathbf{u}}}_{\alpha^+}\right)}_{\alpha\in {\rm\underline{\mathbf{X}}}}$ of the variable $\mathrm{\underline{\mathbf{f}}}_{\alpha^+}$ and an individual $\underline{\bm{\alpha}}$ as values to the variables $\mathrm{\underline{\mathbf{f}}}_{\alpha^+}$ and $\alpha$ respectively, the sentence $[\underline{\bm{\alpha}}\backslash\alpha]
      [\mathrm{\underline{\mathbf{u}}}_{\underline{\bm{\alpha}}^+}\backslash \mathrm{\underline{\mathbf{f}}}_{\alpha^+}] \Psi$ is valid in $\mathcal{M}$.
\end{enumerate}
This defines the validity of the nonstandard formulas that can be deduced from SUM-F in terms of the well-established validity of standard first-order formulas.\hfill$\Box$
\end{Definition}

\begin{Proposition}\label{prop:notLS}
If $\mathfrak{T}$ has a model $\mathcal{M}$, then $\mathcal{M}$ is not countable.
\end{Proposition}
\begin{proof} Suppose $\mathfrak{T}$ has a model $\mathcal{M}$, and $\mathcal{M}$ is countable. That means that there are only countably many subsets of $\mathbb{N} = \{0, 1, 2, \ldots\}$ in $\mathcal{M}$, and that the powerset $\mathcal{P}(\mathbb{N})$ in $\mathcal{M}$ contains those subsets: we thus assume that there are subsets of $\mathbb{N}$ that are ``missing'' in $\mathcal{M}$. Let ${\rm\underline{\mathbf{A}}}$ be any subset of $\mathbb{N}$ that is \textbf{not} in $\mathcal{M}$, and let ${\rm\underline{\mathbf{h}}} \in {\rm\underline{\mathbf{A}}}$. All numbers $0, 1, 2, \ldots$ are in $\mathcal{M}$ (including ${\rm\underline{\mathbf{h}}}$), so for an arbitrary number ${\rm\underline{\mathbf{n}}} \in \mathbb{N}$ there is thus on account of the ur-function axiom (Ax. \ref{ax:UFA}) an ur-function on $\{{\rm\underline{\mathbf{n}}}\}$ that maps ${\rm\underline{\mathbf{n}}}$ to ${\rm\underline{\mathbf{n}}}$ and an ur-function on $\{{\rm\underline{\mathbf{n}}}\}$ that maps ${\rm\underline{\mathbf{n}}}$ to ${\rm\underline{\mathbf{h}}}$. Since $\mathbb{N}$ is in $\mathcal{M}$, we get on account of SUM-F and Nonstandard Universal Elimination that
\begin{equation}\label{eq:notLS}
\models_{\mathcal{M}} {(\forall f_{p^+})}_{p \in \mathbb{N}} \exists F_{\mathbb{N}} \exists Y\left(F_{\mathbb{N}}:\mathbb{N} \twoheadrightarrow Y \wedge {\bigwedge}_{p\in \mathbb{N}} F_{\mathbb{N}}:p\mapsto \imath q(f_{p^+}: p \mapsto q)\right)
\end{equation}
Eq. (\ref{eq:notLS}) being valid in $\mathcal{M}$ means thus that for \emph{any} team assignment $g$, there is a sum function ${\rm\underline{\mathbf{F}}}^g_{\mathbb{N}}$ in $|\mathcal{M}|$ for which
\begin{equation}
\forall p \in \mathbb{N}({\rm\underline{\mathbf{F}}}^g_{\mathbb{N}}:p \mapsto \imath q({\rm\underline{\mathbf{f}}}^g_{p^+}:p\mapsto q ))
\end{equation}
where the variable ${\rm\underline{\mathbf{f}}}^g_{p^+}$ ranges over ur-functions $g(f_{0^+}), g(f_{1^+}), g(f_{2^+}), \ldots$ That said, the crux is that there is a team assignment $g^*$ which assigns to the variables $f_{0^+}, f_{1^+}, f_{2^+}, \ldots$ , the constants $g^*(f_{0^+}), g^*(f_{1^+}), g^*(f_{2^+}), \ldots$ such that for all $n \in \bm{\omega}$ we have
\begin{equation}\label{eq:UrFunctionForA}
\left\{ \begin{array}{l}
g^*(f_{n^+}): n \mapsto n \ \ if \ n \in{\rm\underline{\mathbf{A}}} \\
g^*(f_{n^+}): n \mapsto {\rm\underline{\mathbf{h}}} \ \ if \ n\not\in{\rm\underline{\mathbf{A}}}
\end{array} \right .
\end{equation}
To see that, note that there is (i) at least one team assignment $g$ such that $g(f_{0^+}) = \{\langle 0,0\rangle \}_{\{0\}}$ so that $g(f_{0^+}):0\mapsto 0$, and (ii) at least one team assignment $g^\prime$ such that $g^\prime(f_{0^+}) = \{\langle 0,{\rm\underline{\mathbf{h}}}\rangle \}_{\{0\}}$ so that $g^\prime(f_{0^+}):0\mapsto {\rm\underline{\mathbf{h}}}$: it is therefore a certainty that there is at least one team assignment $g^0$ such that Eq. (\ref{eq:UrFunctionForA}) is satisfied for $n = 0$. Now assume that there is a team assignment $g^k$ such that Eq. (\ref{eq:UrFunctionForA}) is satisfied for $n \in \{0, 1, \ldots, k\}$. There is, then, at least one team assignment $g^k_{k+1}$ such that for $n \in \{0, 1, \ldots, k\}$ we have $g^k_{k+1}(f_{n^+}) = g^k(f_{n^+})$ and such that $g^k_{k+1}(f_{(k+1)^+}) =
\{\langle k+1,k+1\rangle \}_{\{k+1\}}$ so that $g^k_{k+1}(f_{(k+1)^+}):k+1\mapsto k+1$, \textbf{and} there is then at least one team assignment $g^k_{{\rm\underline{\mathbf{h}}}}$ such that for $n \in \{0, 1, \ldots, k\}$ we have $g^k_{{\rm\underline{\mathbf{h}}}}(f_{n^+}) = g^k(f_{n^+})$ and such that $g^k_{{\rm\underline{\mathbf{h}}}}(f_{(k+1)^+}) = \{\langle k+1,{\rm\underline{\mathbf{h}}}\rangle \}_{\{k+1\}}$ so that $g^k_{{\rm\underline{\mathbf{h}}}}(f_{(k+1)^+}):k+1\mapsto {\rm\underline{\mathbf{h}}}$: it is then a certainty that there is at least one team assignment $g^{k+1}$ such that Eq. (\ref{eq:UrFunctionForA}) is satisfied for $n \in \{0, 1, \ldots, k+1\}$. By induction, there is thus a team assignment $g^*$ such that Eq. (\ref{eq:UrFunctionForA}) is satisfied for all $n \in \bm{\omega}$. Ergo, there is a sum function ${\rm\underline{\mathbf{F}}}^*_{\bm{\omega}}$ for which
\begin{equation}
\forall p \in \bm{\omega}({\rm\underline{\mathbf{F}}}^*_{\bm{\omega}}:p \mapsto g^*(f_{p^+})(p))
\end{equation}
But then we get ${\rm\underline{\mathbf{F}}}^*_{\bm{\omega}}[\bm{\omega}] = {\rm\underline{\mathbf{A}}}$, so ${\rm\underline{\mathbf{A}}}$ is in $\mathcal{M}$, contrary to what was assumed. Ergo, if $\mathfrak{T}$ has a model, it is not countable.
\end{proof}

\begin{Remark}\rm
Prop. \ref{prop:notLS}  is a significant result that does not hold in ZF: given Ths. \ref{SEP} and \ref{SUB}, this provides an argument for considering the present theory $\mathfrak{T}$ to be \textbf{stronger} than ZF. The crux here is that the nonstandard sentence (\ref{eq:notLS}), which is an instance of SUM-F, has to be valid in $\mathcal{M}$: the notion of validity of Def. \ref{def:validity} entails that there are uncountably many variables ${\rm\underline{\mathbf{f}}}^g_{p^+}$, ranging over a family of individual ur-functions indexed in $\mathbb{N}$, in the language of the model. As a result, the subsets of $\mathbb{N}$  that can be constructed within the model are non-denumerable---a model of $\mathfrak{T}$ in which Multiple Universal Elimination, inference rule \ref{inf:MUQE}, applies for at most countably many variables ${\rm\underline{\mathbf{f}}}^g_{p^+}$ is thus nonexisting. Thus speaking, the L\"{o}wenheim-Skolem theorem does not apply because $\mathfrak{T}$ is a \textbf{nonstandard} first-order theory, meaning that it is not the case that $\mathfrak{T}$ can be reformulated as a standard first-order theory, nor that $\mathfrak{T}$ is a second order theory---this latter fact will be proven in the next proposition.\hfill$\Box$
\end{Remark}

\begin{Proposition}\label{prop:not2ndOrder}
$\mathfrak{T}$ is not a second-order theory.
\end{Proposition}

\begin{proof}
Let's assume that $\mathfrak{T}$ is a second-order theory. That is, let's assume that the use of a multiple quantifier ${(\forall f_{\{\alpha\}})}_{\alpha\in {\rm\hat{\mathbf{X}}}}$ amounts to second-order quantification. With such a multiple quantifier, we \emph{de facto} quantify over all functional relations on the set ${\rm\hat{\mathbf{X}}}$---note, however, that we do not quantify over all functional relations on the universe of sets! But this has an equivalent in ZF: if ${\rm\hat{\mathbf{X}}}$ is a constant (a set), and ${\rm\hat{\mathbf{Y}}}^{\rm\hat{\mathbf{X}}}$ is the set of all functions from ${\rm\hat{\mathbf{X}}}$ to a set ${\rm\hat{\mathbf{Y}}}$, then with the quantifier $\forall B \forall f\in B^{\rm\hat{\mathbf{X}}}$ in a sentence
\begin{equation}\label{eq:finZF}
\forall B \forall f\in B^{\rm\hat{\mathbf{X}}} \Psi
\end{equation}
we \emph{de facto} quantify over all functional relations on the set ${\rm\hat{\mathbf{X}}}$ too. Ergo, if $\mathfrak{T}$ is a second-order theory, then ZF is a second-order theory too. But ZF is a first-order theory, and not a second-order theory. So by modus tollens, $\mathfrak{T}$ is not a second-order theory.
\end{proof}
\noindent As an additional heuristic argument, we can also directly compare second-order quantification and the present nonstandard first-order quantification with a multiple quantifier ${(\forall f_{\{\alpha\}})}_{\alpha\in {\rm\hat{\mathbf{X}}}}$. Let's first look at second-order quantification with a quantifier $\forall {\it \Phi}$ where the variable ${\it \Phi}$ ranges over functional relations. An arbitrary individual functional relation $\hat{\bm{\Phi}}$ has the entire proper class of things as its `domain', so $\hat{\bm{\Phi}}$ corresponds to a proper class of ur-functions: for an arbitrary thing ${\rm\hat{\bm{\alpha}}}$ there is a $\hat{\bm{\Phi}}$-related ur-function ${\rm\hat{\mathbf{u}}}_{{\rm\bm{\hat{\alpha}}}^+}$ for which ${\rm\hat{\mathbf{u}}}_{{\rm\bm{\hat{\alpha}}}^+}: {\rm\hat{\bm{\alpha}}}\mapsto \imath\beta\hat{\bm{\Phi}}({\rm\hat{\bm{\alpha}}},\beta)$. A quantifier $\forall {\it \Phi}$ is thus equivalent to a \textbf{proper class} of simple quantifiers $f_{{\rm\bm{\hat{\alpha}}}^+}$ ranging over ur-functions on the singleton of a thing ${\rm\bm{\hat{\alpha}}}$. The universe of $\mathfrak{T}$, however, does not contain a set ${\rm\hat{\mathbf{U}}}$ of all things, so there is no multiple quantifier ${(\forall f_{\alpha^+})}_{\alpha \in {\rm\hat{\mathbf{U}}}}$ which would be equivalent to quantifier $\forall {\it \Phi}$: a multiple quantifier ${(\forall f_{\alpha^+})}_{\alpha \in {\rm\hat{\mathbf{X}}}}$ is \textbf{at most} equivalent to an \emph{infinite set} of simple quantifiers $f_{{\rm\bm{\hat{\alpha}}}^+}$ and the degree of infinity is then bounded by the notion of a set. Thus speaking, since a set does not amount to a proper class, a multiple quantifier ${(\forall f_{\alpha^+})}_{\alpha \in {\rm\hat{\mathbf{X}}}}$ does not amount to second-order quantification. See the figure below for an illustration.\\
\ \\ 
\begin{figure}[h!]
\centering
\hfill
\subfigure[$2^{\rm nd}$ order]{\includegraphics[width=0.495\textwidth]{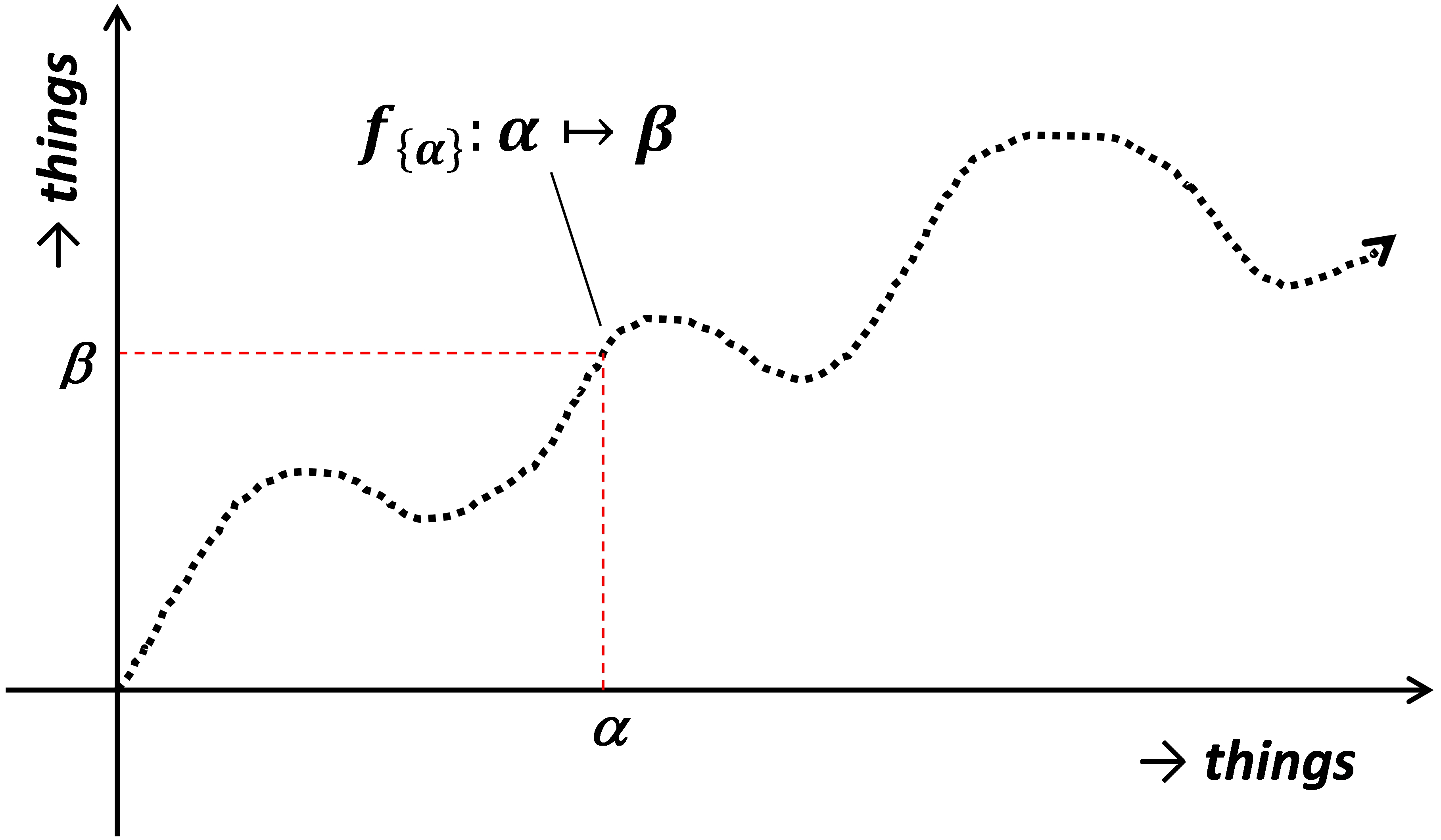}}
\hfill
\subfigure[nonstandard $1^{\rm st}$ order]{\includegraphics[width=0.495\textwidth]{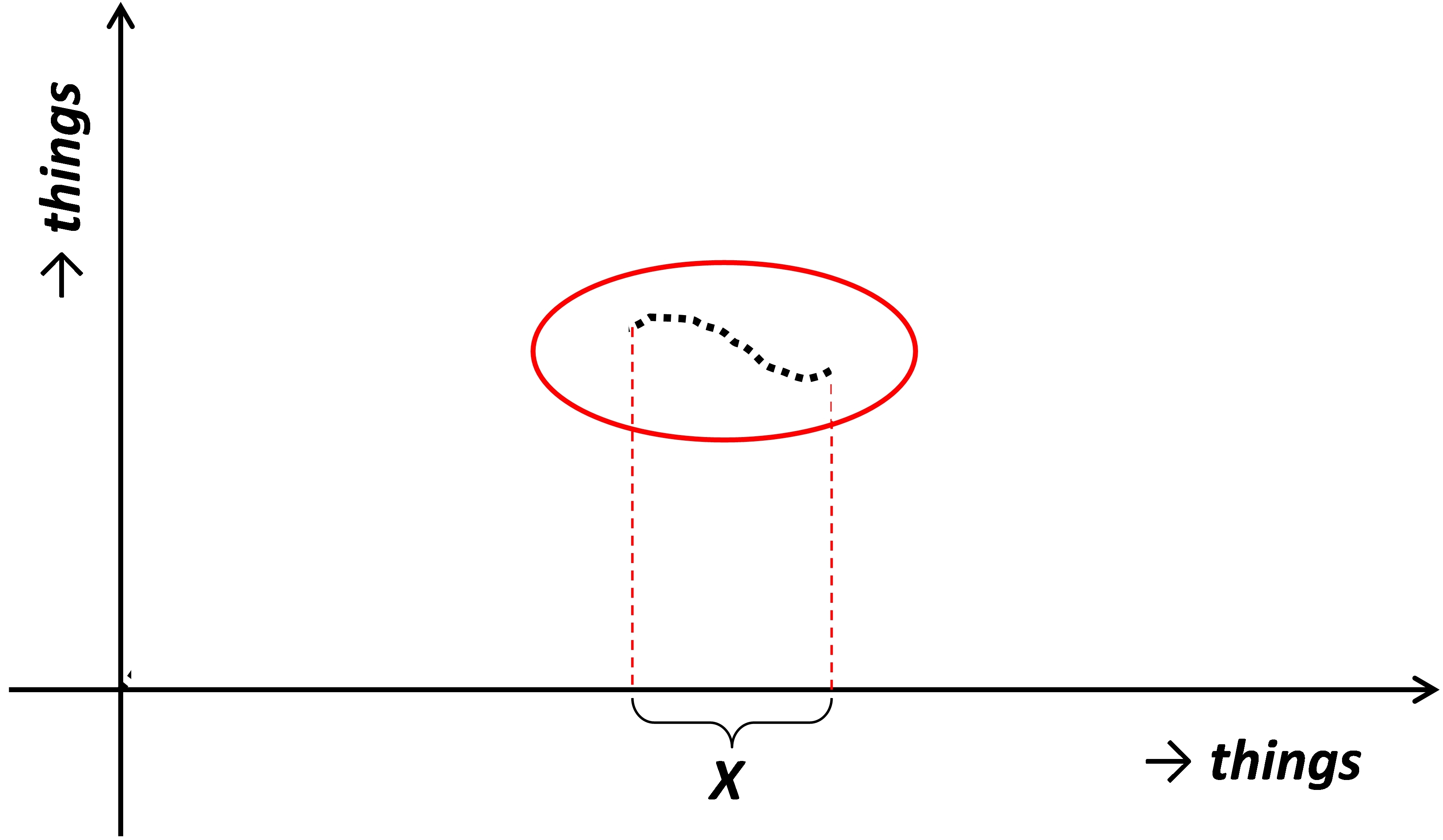}}
\hfill
\caption{
Illustration of the heuristic argument. In both diagrams (a) and (b), all things in the universe of $\emph{T}$ are \emph{for illustrative purposes} represented on the horizontal and vertical axes. In diagram (a), the dotted black line represents an arbitrary functional relation $\hat{\bm{\Phi}}$: each dot corresponds to a constant ur-function as indicated, so the dotted line is equivalent to a \underline{proper class} of ur-functions. In diagram (b) it is indicated of which things on the horizontal axis the set ${\rm\hat{\mathbf{X}}}$ is made up, and each of the black dots within the red oval corresponds to a constant ur-functions: the dotted line segment is thus equivalent to a \underline{set} of ur-functions. So, a multiple quantifier ${(\forall f_{\alpha^+})}_{\alpha \in {\rm\hat{\mathbf{X}}}}$ cannot be equivalent to a quantifier $\forall {\it \Phi}$.}\end{figure}

\subsection{The axioms of category theory}
In Def. \ref{def:universe} it has been assumed that the universe of sets and functions is a category. In this section we prove that the axioms of category theory for the arrows indeed hold for the functions (which are the arrows of the present category). That means that we must prove the following:
\begin{enumerate}[label=(\roman*)]
\item that domain and codomain of any function on any set are unique;
\item that, given sets $X$ and $Y$ and functions $f_X$ and $g_Y$ with $Y= f_X[X]$, there is a function $h_X = g_Y\circ f_X$ such that $h_X$ maps every $\alpha\in X$ to the image under $g_Y$ of its image under $f_X$;
\item that for any set $X$ there is a function $1_X$ such that $f_X \circ 1_X = f_X$ and $1_{f_X[X]}\circ f_X = f_X$ for any function $f_X$ on $X$.
\end{enumerate}
\paragraph{Ad(i): domain and codomain of a function $f_X$ are unique}\ \\
This has already been proven in Sect. \ref{sect:functions}. GEN-F (Ax. \ref{ax:GEN-F}) guarantees that for any set $X$, any function $f_X$ has at least one domain and at least one codomain. Ax. \ref{ax:NEGD} guarantees that no other thing (set or function) than $X$ is a domain of $f_X$. And Ax. \ref{ax:NEGC} guarantees that no other thing (set or function) than the image set $f_X [X]$ is a codomain of $f_X$. That proves uniqueness of domain and codomain.
\paragraph{Ad(ii): existence of the composite of two functions}\ \\
Given a set $X$, a function $f_X$ and a function $g_Y$ with $Y= f_X[X]$, there is for every $\alpha \in X$ precisely one ur-function $h_{\alpha^+}$ such that
\begin{equation}
h_{\alpha^+}: \alpha \mapsto \imath\beta(g_Y:f_X(\alpha)\mapsto\beta)
\end{equation}
So, there is a sum function $H_X$ that maps each $\alpha \in X$ precisely to its image under the ur-function $h_{\alpha^+}$ for which $\imath\xi(h_{\alpha^+}: \alpha \mapsto \xi) = \imath\beta(g_Y:f_X(\alpha)\mapsto\beta)$. This sum function is precisely the composite $H_X = g_Y\circ f_X$. The proof that function composition is associative is omitted.
\paragraph{Ad(iii): existence of an identity function on any set $X$}\ \\
Given a set $X$, there is for every $\alpha \in X$ precisely one ur-function $1_{\alpha^+}$ for which
\begin{equation}
1_{\alpha^+}: \alpha \mapsto \alpha
\end{equation}
Therefore, there is a sum function $F_X$ that maps every $\alpha \in X$ precisely to $\imath\xi(1_{\alpha^+}: \alpha \mapsto \xi) = \alpha$. This sum function is the requested function $1_X$. The proof that this sum function $1_X$ satisfies the properties that $g_X \circ 1_X = g_X$ and $1_{g_X[X]}\circ g_X = g_X$ for any function $g_X$ on $X$, is omitted.\\
\ \\
This shows that the axioms for a category hold for the proper class of functions.

\subsection{Concerns regarding inconsistency}

We address the main concerns regarding inconsistency, which are in particular that the existence of a set of all sets or a set of all functions can be derived from $\mathfrak{T}$.

\begin{Conjecture}
The category of sets and functions does not contain a set of all sets.
\end{Conjecture}
\paragraph{Heuristic argument:} A set ${\rm\hat{\mathbf{U}}}$ of all sets does not exist (i) because REG (Ax. \ref{ax:REG}) excludes that ${\rm\hat{\mathbf{U}}}$ exists \emph{a priori}, and (ii) because SUM-F, the only constructive axiom of $\mathfrak{T}$ that is not a theorem of ZF, excludes that ${\rm\hat{\mathbf{U}}}$ exists \emph{by construction}. The crux is that one must \emph{first} have constructed the set $X$ \emph{before} one can construct a sum function $F_X$: the set $X$ is thus a regular set, and by applying SUM-F one cannot create a new set with a higher cardinality than the set $X$ because the graph of $F_X$ contains precisely one element for each element of $X$. The same for the image set $F_X[X]$: it cannot have a higher cardinality than $X$. Therefore, SUM-F doesn't allow the construction of a set ${\rm\hat{\mathbf{U}}}$ of all sets.\hfill$\Box$
\begin{Conjecture} 
The category of sets and functions on sets does not contain a set $\hat{\bm{\Omega}}$ of all functions.
\end{Conjecture}
\paragraph{Heuristic argument:} Suppose that we have a set $\hat{\bm{\Omega}}$ such that
\begin{equation}
\forall X \forall f_X (f_X \in \hat{\bm{\Omega}})
\end{equation}
Then on account of Th. \ref{SEP} we can single out the subset $\hat{\bm{\Omega}}_1$ of \emph{all} identity functions:
\begin{equation}
\forall\alpha(\alpha\in\hat{\bm{\Omega}}_1 \Leftrightarrow \alpha \in \hat{\bm{\Omega}} \wedge \exists X(\alpha = 1_X))
\end{equation}
But then there also exists the identity function $1_{\hat{\bm{\Omega}}_1}$, for which  $1_{\hat{\bm{\Omega}}_1}:{\hat{\bm{\Omega}}_1}\twoheadrightarrow{\hat{\bm{\Omega}}_1}$, $1_{\hat{\bm{\Omega}}_1}:1_{\hat{\bm{\Omega}}_1}\mapsto1_{\hat{\bm{\Omega}}_1}$. This latter feature that $1_{\hat{\bm{\Omega}}_1}$ maps itself to itself contradicts the axiom of regularity for functions, REG-F (Ax.\ref{ax:REGF}). Ergo, there is no set $\hat{\bm{\Omega}}$ of all functions.\hfill$\Box$

\section{Conclusions}
\noindent
The main conclusion is that the aim stated in the introduction has been achieved: a theory $\mathfrak{T}$, with a vocabulary containing countably many constants, has been introduced which lacks the two unwanted c.q. pathological features of ZF. Each axiom of $\mathfrak{T}$ is a typographically finite sentence, so contrary to what's the case with infinitary logics, each axiom can be written down explicitly. But not just that: $\mathfrak{T}$ is finitely axiomatized, so contrary to what's the case with ZF, the entire theory $\mathfrak{T}$ can be written down explicitly on a piece of paper. In addition, it has been shown that $\mathfrak{T}$, contrary to ZF, does not have a countable model---if it has a model at all, that is. This failure of the downward L\"{o}wenheim-Skolem theorem for $\mathfrak{T}$ is due to the nonstandard nature of $\mathfrak{T}$.

Furthermore, three reasons can be given as to why $\mathfrak{T}$ might be potentially applicable as a foundational theory for mathematics. First of all, it has been proven that the axioms of ZF, translated in the language $L_{\mathfrak{T}}$ of $\mathfrak{T}$, can be derived from  $\mathfrak{T}$. While we acknowledge that this result does not automatically imply that all theorems about sets derived within the framework of ZF are necessarily also true in the framework of $\mathfrak{T}$ because in the latter framework sets exist whose elements aren't sets, it nevertheless shows that the tools available in the framework ZF for constructing sets are also available in the framework of $\mathfrak{T}$. Secondly, it has been proved that the axioms of a category hold for the universe of discourse that is associated to $\mathfrak{T}$, which is a category of sets and functions: this universe might then serve as the ontological basis for the various (large) categories studied in category theory. Thirdly, $\mathfrak{T}$ is easy to use in everyday's mathematical practice because for any set $X$ we can construct a function $f_X$ by giving a defining function prescription $f_X:X\twoheadrightarrow Y\ , \ f_X: \alpha\overset{\rm def}{\longmapsto}\iota\beta\Phi(\alpha,\beta)$ where $\Phi$ is some functional relation: $\mathfrak{T}$ then guarantees that $F_X$ exists, as well as its graph, its image set, and the inverse image sets for every element of its codomain---ergo, giving a defining function prescription is a tool for constructing sets.

On the other hand, the present theory $\mathfrak{T}$ gives immediately rise to at least three purely aesthetical arguments for rejection. First of all, the universe of $\mathfrak{T}$ contains sets and functions, the latter being objects sui generis: this entails a departure from the adage `everything is a set' that holds in the framework of ZF(C), and that will be enough to evoke feelings of dislike among mathematical monists who hold the position that set theory, in particular ZF or ZFC, has to be the foundation for mathematics. Secondly, although the universe of $\mathfrak{T}$ is a category, the formal language of $\mathfrak{T}$ contains $\in$-relations $t_1 \in t_2$ as atomic formulas: an $\in$-relation is, thus, not reduced to a mapping in the language of category theory, and that fact alone will be enough to evoke feelings of dislike among mathematical monists who hold the position that category theory has to be the foundation for mathematics. Thirdly, the language of $\mathfrak{T}$ entails a rather drastic departure from standard first-order language: that will be enough to evoke feelings of dislike among those who attach a notion of beauty to the standard first-order language of ZF(C), or who consider that the language of category theory is \emph{all of} the language of mathematics.

That said, since $\mathfrak{T}$ is a nonstandard theory there is the obvious risk that $\mathfrak{T}$ is not (relatively) consistent. It is true that we have argued that the category of sets and functions cannot contain a set of all sets nor a set of all functions, but it remains the case that further research may reveal that $\mathfrak{T}$ has unintended consequences which render it inconsistent. Furthermore, a limitation of this study is that the axiom of choice has been left out. We can easily express AC in the language $L_T$ as
\begin{equation}\label{eq:AC}
  \forall X\neq\emptyset( \Theta(X) \Rightarrow \exists f_X \forall Z\in X \forall\gamma(f_X:Z\mapsto\gamma \Rightarrow \gamma \in Z)
\end{equation}
where $\Theta(X)$ stands for
\begin{equation}
\forall \alpha\in X\exists Y(\alpha = Y \wedge \exists \eta(\eta\in Y))\wedge \forall U\in X\forall V\in X\not\exists\beta(\beta\in U \wedge\beta\in V)
\end{equation}
But the question is whether this has to be added as an axiom, or whether it can be derived as a theorem of $\mathfrak{T}$---we certainly have for any $Z\in X$ that there is an ur-function $f_{Z^+}$ such that $\imath\xi(f_{Z^+}:Z\mapsto\xi)\in Z$. We leave this as a topic for further research. Another limitation of the present study is that it has not been investigated whether the calculus has any of the various soundness and completeness properties. It may then very well turn out that minor details in the definitions for the nonstandard syntax and semantics require some revisions, but this is left as a topic for further research---(dis-)proving that these properties hold is a sizeable research project in itself. Of course, contrary to the aesthetic arguments mentioned above, which can be dismissed as nonmathematical, negative results in that direction may yield a serious, mathematical argument to reject the present nonstandard theory as a foundation for mathematics.

We therefore cannot but conclude this paper with the cliche that further research is necessary: additional results are needed to establish whether the nonstandard theory $\mathfrak{T}$ introduced in this paper constitutes an advancement in the foundations of mathematics. The proven fact that $\mathfrak{T}$ lacks the pathological features of ZF may provide a reason for such further research, but it is emphasized that it may turn out to be a dead end. That is to say: the present marriage of set theory and category theory may look promising from a certain perspective, but it still may end in divorce. 

\paragraph{\emph{\textbf{Acknowledgements}}} The author wants to thank Harrie de Swart (Erasmus University Rotterdam), Jean Paul van Bendegem and Colin Rittberg (both Free University of Brussels) for their useful comments. This research has been facilitated by the Foundation Liberalitas (the Netherlands)

\end{document}